\documentclass[12pt]{amsart}

\pdfoutput=1
\usepackage{amsmath,amssymb}
\usepackage{amsfonts}
\usepackage{amsthm}
\usepackage{latexsym}
\usepackage{graphicx}
\usepackage{epsfig}
\usepackage{enumitem}
\usepackage{xcolor}
\usepackage[normalem]{ulem}

\newtheorem{theorem}{Theorem}[section]
\newtheorem{corollary}[theorem]{Corollary}

\newtheorem{lemma}{Lemma}[section]
\newtheorem{proposition}[lemma]{Proposition}
\newtheorem{regtheorem}[lemma]{Theorem}
\newtheorem{regcorollary}[lemma]{Corollary}

\newtheorem{claim}{Claim}[section]
\newtheorem{definition}[claim]{Definition}
\newtheorem{remark}[claim]{Remark}

\theoremstyle{definition}
\newtheorem{example}[lemma]{Example}
\theoremstyle{plain}

\makeatletter \@addtoreset{equation}{section} \makeatother

\def\ddt{\frac{d}{dt}}

\def\RR{{\mathrm R}}
\def \2R{{\hat{\RR}}}

\def\Rc{{\mathrm {Rc}}}

\def\SS{{\mathrm S}}

\def\He{\mathrm {Hess}}

\def\id{\mathrm{Id}}

\def\lie{\mathcal{L}}
\def\xi{\partial_{x_i}}

\def\tr{\text{tr}}

\usepackage[margin=1.19 in]{geometry}

\begin{document}

\title{K\"{a}hler Solitons, Contact Structures, and Isoparametric Functions}

\author{Hung Tran$^*$}
\address{Department of Mathematics and Statistics,	Texas Tech University, Lubbock, TX 79409}
\email{hung.tran@ttu.edu}


\begin{abstract} Many known examples of gradient K\"{a}hler-Ricci soliton in real dimension four are toric and the symmetry is intrinsically related to the potential function $f$ and the scalar curvature $\SS$. In this article, we consider the case that $f$ and $\SS$ are functionally dependent and deduce a complete classification, while the independence case is addressed elsewhere. The main theorem recovers all known examples of cohomogeneity one symmetry. We also discover a connection to the theory of isoparametric functions and contact geometry. Indeed, a key ingredient is a new characterization for a deformed Sasakian structure generalizing a classical result.

\end{abstract}
\maketitle

\section{Introduction}

A gradient Ricci soliton (GRS) $(M, g, f, \lambda)$ is a Riemannian manifold with metric $g$, potential function $f$, and a constant $\lambda$ such that, for $\Rc$ denoting the Ricci curvature,  
\begin{equation}
	\label{grs}
	\Rc+\text{Hess}{f}=\lambda g. 
\end{equation}
Such a structure is a self-similar solution to the Ricci flow and plays a crucial role in its analysis. The general theory was introduced by R. Hamilton \cite{H3} and has several celebrated applications including \cite{perelman1, perelman2, bohmwilking, bs091, bs072, CGT21}. Recently, there have been significant breakthroughs towards a Ricci flow program in higher dimensions \cite{brendle18highsurgery, bamler2021compactness, bamler2021structure, bamler2021fundamental, tian1,  ST17KRf, DS20krf, HL24algebraicunique, HJST24}. Dimension four seems to be the most tractable and it is of tremendous interest to investigate corresponding GRS.  


A gradient K\"{a}hler Ricci soliton (GKRS) $(M, g, f, J, \lambda)$ is a GRS such that $(M, g, J)$ is K\"{a}hler for a complex structure $J$. In particular, the soliton vector field $\nabla f$ is automatically holomorphic.
The subject has an extensive literature
and recent efforts lead to the classification of all GKRS surfaces with $\lambda>0$ \cite{CDSexpandshriking19, CCD22finite, BCCD22KahlerRicci, LW23}. Thus, there is realistic hope to obtain a full classification of all GKRS in real dimension four. Indeed, {many} known examples have toric symmetry which is intrinsically related to the potential function $f$ and the scalar curvature $\SS$. Here we propose a strategy based on the relation between these functions and applicable to all signs of $\lambda$. It turns out that there will be a dichotomy as described below. 

\begin{definition} Two functions are functionally dependent if their gradients are scalar multiples of each other at points they are non-zero. Otherwise, they are functionally independent.  
\end{definition}
In a companion article \cite{tran24toric}, we consider the independence case and prove that the soliton admits a torus action under a generic assumption. In this paper, we focus on the dependence case and deduce a complete classification. Indeed, on a GRS, $f$ and $\SS$ are functionally dependent if and only if $\Delta f$ and $|\nabla f|$ are constant on each connected component of a level set of $f$ (see Proposition \ref{rectequivalent}). This property generally defines an isoparametric function; the study of such functions in space forms was motivated by questions in geometric optics \cite{som1918, cartan38, lv37}. The classification in an ambient round sphere, Question 34 in S. T. Yau's list \cite{Yauopen93}, is remarkably deep and has attracted enormous interest by, for example, \cite{nomizu73, Munzner80iso, Abresch83foursix, CCJ07four, chi20four}. Certain aspects of the theory is extended to a general Riemannian manifold \cite{wang87iso, gt13exotic, miyaoka13, QT15, derd21} and will contribute to our analysis.

In the study of a GKRS, the early constructions of examples with maximal symmetry \cite{koi90, caohd96, CV96, caohd97limits, fik03, PTV99quasi}, all have isoparametric potential functions. Without the K\"{a}hler setup, an equivalent condition is considered in \cite{kim17}. It also arises as a consequence of constant scalar curvature \cite{FG16csc, CZ2021rigidity}. Furthermore, in \cite[Section 3]{CDSexpandshriking19}, the authors discovered that an appropriate rescaling of the potential function, on an expanding GKRS, converges to an isoparametric function with respect to a suitable metric. Our first result classifies complete K\"{a}hler GRS surfaces with such a potential function. 

\begin{theorem}
	\label{main1}
	Let $(M, g, J, f)$ be a complete connected non-flat K\"{a}hler GRS of real dimension four with $f$ non-constant. If $f$ and $\SS$ are functionally dependent then the manifold must be either 
	\begin{itemize} 
		\item a quotient of a product of a constant curvature surface with a $2D$ K\"{a}hler GRS, 
		\item foliated by equidistant level sets of f. Each regular level set is a quotient, by a discrete group, of a deformed Sasakian space-form in dimension three.
	\end{itemize}
	
\end{theorem}  
\begin{remark} Lifting to an universal cover, there are three models of simply connected Sasakian structures with constant holomorphic sectional curvature \cite{tanno69sasa}. Thus, the soliton equation (\ref{grs}) reduces to ODEs which can be solved explicitly. Furthermore, if $\lambda\geq 0$, then there is a singular orbit (Theorem \ref{HDcontact}), the Sasakian structure must be on a sphere, the group is $U(2)$, and the singular orbit is a point or round $\mathbb{CP}^1$. Then, if $M$ is non-compact, $M$ is either $\mathbb{C}^2$ or a holomorphic line bundle $L^\ell$, $\ell \in \mathbb{Z}$, over $\mathbb{CP}^1$. $\ell$ corresponds to the quotient by $\mathbb{Z}_{|\ell|}$ of each principal orbit and also denotes the first Chern class of the bundle. By \cite[Lemma 1.2]{fik03}, one obtains a shrinking/steady/expanding soliton depending on whether $\ell$ is greater/equal/less than $-2$. If $\lambda<0$, then each Sasakian space-forms could arise and the analysis is given in \cite[Theorem 4.20]{DW11coho}. Moreover, there may be no singular orbit at all with the level sets collapsing at infinite distance case (iii) of Theorem \ref{HDcontact} and Example \ref{cuspexample}. Hence, the classification includes all $U(2)$-invariant metrics constructed by \cite{koi90, caohd96, CV96, caohd97limits, fik03, PTV99quasi}.    
	
In relation to the recent classification of  shrinking GKRS in real dimension four, our assumption recovers all models except for $T^2$-symmetry metrics on $Bl_2(\mathbb{CP}^2)$ and $Bl_1(\mathbb{C}\times \mathbb{CP}^1)$. They correspond to the case $f$ and $\SS$ are functionally independent. 
\end{remark}
\begin{remark}
The dimension assumption is crucial as \cite{DW11coho} describes a more general construction in higher dimensions such that $f$ and $\SS$ are functionally dependent.
\end{remark}
 
A soliton structure on a manifold can always be lifted to its universal cover and we have a refined statement.

\begin{corollary}
	\label{simply}
	Let $(M, g, J, f)$ be a {complete} simply connected irreducible non-compact GKRS of real dimension four{ with $f$ non-constant}. If $f$ and $\SS$ are functionally dependent then the manifold must be of cohomogeneity one{ with isometry group of dimension four}. {Thus, the GKRS is of maximal irreducible symmetry. Moreover, exactly one of the following holds:}
	
	\begin{enumerate}[label=(\alph*)]
		\item there is exactly one singular orbit, and it is simply connected;
		\item $\lambda<0$, $f$ has no critical point, all orbits are principal, and $M\cong\mathbb{R}\times P$, where each level set of $f$ is a deformation of the standard Sasakian structure on $P=\widetilde{SL}_2(\mathbb{R})$ or $P=\mathrm{Nil}^3$.
	\end{enumerate}
	In particular, (a) holds if $\lambda\geq 0$.

\end{corollary}
\begin{remark}
	See \cite{tran23kahler} for the classification of GKRS with such symmetry. 
\end{remark}


Furthermore, there is a partial converse.
\begin{corollary}
	\label{coho1}
	Let $(M, g, J, f)$ be a complete connected irreducible K\"{a}hler GRS surfaces of cohomogeneity one in real dimension four{ with $f$ non-constant}. If either $\lambda\neq 0$ or the isometry group is compact, then $f$ and $\SS$ are functionally dependent and each principal orbit is a connected deformed homogeneous Sasakian structure.  
\end{corollary}


The results above suggest a connection between a GKRS and a Sasakian structure, a fundamental notion of contact geometry. This arguably is an odd-dimensional counterpart of symplectic geometry and both lie at the heart of classical mechanics. 
Thus, it has rich literature; for example, contact transformation was studied by S. Lie \cite{liebook1896}. Here we focus on a direction with more attention to an associated Riemannian metric. The current theory owns much to the foundational work of S. Sasaki, D. Blair, and others \cite{sasaki60, SH62, sas71, blair70, blair76}; see also a recent book \cite{BGbookSasakian08} and a survey \cite{blair19}. 

An odd-dimensional Riemannian manifold $(P, g_P)$- with a vector field $\zeta$, a $1$-form $\eta$, and a tensor field of type $(1, 1)$ $\Phi$- is called an almost contact metric structure if 
  	\[\eta(\zeta)=1, ~~\Phi^2=-\id+\zeta\otimes \eta, \text{  and } g_P(\Phi(X), \Phi(Y))=g_P(X, Y)-\eta(X)\eta(Y).\]
An almost contact metric structure $(P, \zeta, \eta, \Phi, g)$ is called a contact metric structure if 
	\[ d\eta(X, Y)=g(X, \Phi(Y)).\]
Moreover, a Sasakian structure is a contact metric such that the cone $(M\times \mathbb{R}^+, r^2 g+ dr^2)$ is K\"{a}hler. A deformed contact/Sasakian structure is obtained by a specific rescaling of the Reeb vector field and the transverse metric; see Definition \ref{Sasadeform}.



Let $(M, g, J, f)$ be a K\"{a}hler GRS. Let $M_c$ be a regular level set of $f$ and 
\begin{align}
	\label{standardconstruction}
	V:&=\frac{\nabla f}{|\nabla f|},~~ g_c:= g_{\mid M_c}, ~~\zeta_c:=-J(V),~~\eta_c(\cdot):= g(\cdot, \zeta_c), ~~ \Phi_c (\cdot):= -\eta_c(\cdot) V+ J(\cdot).  \end{align}
Together, $(M_c, \zeta_{c}, \eta_{c}, \Phi_{c}, g_{c})$ is an almost contact metric structure. Our next theorem shows that going from almost contact to deformed contact imposes significant restriction on the soliton structure leading to a full classification.


\begin{theorem}
	\label{HDcontact}
	Let $(M, g, J, f, \lambda)$ be a complete connected K\"{a}hler GRS with $f$ non-constant and  $n=\dim_{\mathbb{C}}M-1>0$. For each regular value $c$, suppose that $(M_c, \zeta_{c}, \eta_{c}, \Phi_{c}, g_{c})$ is a deformed contact structure. Then the soliton is totally determined by a connected Sasakian model $(P, \eta, \zeta, \Phi, g_P)$ whose transverse K\"{a}hler metric $g^\perp$ is Einstein, $\Rc^\perp=k g^\perp$. That is,
	\[g_P= \eta\otimes \eta +g^\perp, ~~~ d\eta=g^\perp(\cdot, \Phi(\cdot)).\]
	There is an interval $I$ with coordinate $s$ such that there is a diffeomorphism $\phi: I\times P \mapsto M_o$, the dense open set of regular points of $f$,  $f\circ \phi= Bs+C$, and  
	\begin{equation} 
		\label{calabi}
		\phi^\ast g= \frac{ds^2}{\alpha(s)}+ \alpha(s) \eta\otimes\eta + (2s+A) {g^\perp}.\end{equation}
	Here $A, B, C$ are constant, $B\neq 0$, and $\alpha$ solves a first order equation, namely  
	\[\lambda(2s+A)=k-\frac{d\alpha}{ds}-\frac{2n\alpha}{2s+A}+B\alpha. \]

	At the lower end point $s_-$ of $I$, $\alpha\rightarrow 0$, and exactly one of the following holds.
	\begin{enumerate}[label=(\roman*)]
		\item (Nut) $2s_-+A=0$ and $k>0$. The end is a single point at finite distance, $(P, \eta, \zeta, \Phi, g_P)$ is, up to a transverse homothety, the standard Sasakian sphere $\mathbb{S}^{2n+1}$, the orbits of $\zeta$ are the fibers of the Hopf fibration $\mathbb{S}^{2n+1}\to\mathbb{CP}^n$, and $g^\perp$ is the pull-back of a multiple of the Fubini-Study metric.
		\item (Bolt) $2s_-+A>0$ and $\frac{d\alpha}{ds}(s_-)>0$. The end is at finite distance, and there is a Riemannian submersion $\pi: P\to N$, with circle fibers, over a K\"{a}hler-Einstein manifold $(N, g_N, J_N)$ with $\Rc_N=k g_N$, $g^\perp=\pi^\ast g_N$ and $d\eta=\pi^\ast\omega_N$ (this $g_N$ is $\check{g}_N$ in Lemma \ref{fibercollapse}); the critical level set of $f$ at this end is $N$, with the metric $(2s_-+A)g_N$.
		\item (Cusp) $\lambda<0$, $k=\lambda(2s_-+A)\leq 0$ and $\frac{d\alpha}{ds}(s_-)=0$. The end is at infinite distance, $B<0$, $I=(s_-, \infty)$, $M=M_o\cong \mathbb{R}\times P$, and $f$ is bounded above and has no critical point.
	\end{enumerate}
	The upper end point $s_+$ of $I$ is either $+\infty$, at infinite distance, or a bolt: $\alpha(s_+)=0>\frac{d\alpha}{ds}(s_+)$, the end is at finite distance, and the conclusions of (ii) hold with $s_+$ in place of $s_-$. In the latter case the lower end is a nut or a bolt, $\lambda>0$ and $k>0$.

\end{theorem}
\begin{remark}
	The main content of the theorem is about $|\nabla f|$ constant on each connected component of a level set of $f$ (rectifiable), constant Ricci, and the explicit Sasakian structure. 
	Equation (\ref{calabi}) is a variant of the Calabi ansatz \cite{Calabi82}. There are also other studies which construct K\"{a}hler GRS from Sasaki-Einstein manifolds \cite{fw11, blair19}. The trichotomy (i)--(iii) is proved in Lemmas \ref{pointsingularorbit}--\ref{ends}. Case (iii) does occur: Example \ref{cuspexample} gives complete, simply connected, irreducible expanding solitons of this type with $P=\widetilde{SL}_2(\mathbb{R})$ and $P=\mathrm{Nil}^3$. There $f$ has no critical point, and the orbits of $\zeta$ are lines, so that $P$ is not a circle bundle. 
	\end{remark}



Indeed, the proofs of Theorem \ref{main1} and \ref{HDcontact} both rely on understanding a deformed contact structure. Thus, the following result might be of independent interest. 

\begin{theorem}
	\label{chardeformed}
	An almost contact metric manifold $(P, \zeta, \eta, \Phi, g)$ satisfies
	\begin{equation*}
		\label{newcon}
		(\nabla_X \Phi) (Y) =  bg(X, Y)\zeta- \eta(Y)bX.
	\end{equation*} 
	for every vector fields $X$ and $Y$ and a real constant $b$ if and only if 
	\begin{enumerate} [label=(\roman*)]
		\item for $b\neq 0$, it is a deformed Sasakian structure;
		\item for $b=0$, $\zeta$ is parallel and $(P, g)$ is locally isometric to a Riemannian product of an interval, tangent to $\zeta$, and a K\"{a}hler manifold whose complex structure is induced by $\Phi$. If, moreover, $g$ is complete, then the universal cover of $(P, g)$ is a product of a line with a simply-connected K\"{a}hler manifold whose almost complex structure is induced by $\Phi$.
	\end{enumerate}	
\end{theorem}
\begin{remark}
The case $b=1$ is a classical result \cite[Theorem 7.3.16]{BGbookSasakian08}.   
\end{remark}
Theorem \ref{chardeformed} has an important consequence. Let $(M, g, J)$ be a {K}\"{a}hler manifold and $P$ a real hypersurface with unit normal $V$ and shape operator $L X =\nabla_X V.$ Let $\zeta=-JV$, $\eta$ the dual 1-form, and $\Phi(\cdot):= J(\cdot)-\eta(\cdot)V.$ It is immediate that $(P, \zeta, \eta, \Phi, g)$ is an almost contact metric structure. 
\begin{theorem}
	\label{hyperdeformedSasa}
	$L=\alpha \text{Id}+\beta \zeta \otimes \eta$ for a constant $\alpha$ iff $(P, \eta, \zeta, \Phi, g)$ is
	\begin{enumerate}[label=(\roman*)]
		\item for $\alpha\neq 0$, a deformed-Sasakian structure. 
		\item for $\alpha=0$, a structure with $\nabla\Phi=0$; precisely, $\zeta$ is parallel and $P$ is locally isometric to a Riemannian product $(a, b)\times N$, where $\zeta$ is tangent to the first factor and $\Phi$ is the pull-back of the complex structure of a K\"{a}hler manifold $(N, J_N)$.
	\end{enumerate}
	
\end{theorem}

The organization of the paper is as follows. Section \ref{preliminary} will recall definitions and preliminary results. The proofs of Theorems \ref{chardeformed} and \ref{hyperdeformedSasa} will be given in Section \ref{defornedsasakian}: using the theory of foliation and toolkit developed by B. O'Neill \cite{ONeill66} to track how the covariant derivative of $\Phi$ changes under a deformation. Section \ref{kahlerandcontac} investigates the case each level set of a potential function $f$ is a deformed contact structure. Then the metric can be written in an explicit way and analysis of the metric's smoothness {and of its ends} leads to Theorem \ref{HDcontact}. Finally, the proofs of Theorem \ref{main1}, Corollaries \ref{simply} and \ref{coho1} will be given in Section \ref{dim4} combining our earlier developments.     

\subsection{Acknowledgment}  We benefit greatly from discussion with Profs. Ronan Conlon, McKenzie Wang, Catherine Searle, and Detang Zhou. We are also grateful to anonymous referees for constructive comments and suggestions.   

\section{Preliminaries}
\label{preliminary}

In this section, we recall preliminary results and certain observations which will be used throughout the article. 
  
\subsection{GRS and K\"{a}hlerity}
Here we recall the definition of a gradient Ricci soliton, the K\"{a}hler setup, and some identities. A Riemannian manifold $(M^n, g)$ with a potential function $f$ is called a GRS if, for $\Rc$ denoting the Ricci curvature and $\lie$ the Lie derivative, 
\begin{equation*}
	\text{Rc}+\frac{1}{2}\lie_{\nabla f} g =\Rc+ \He f= \lambda g.
\end{equation*}
We also recall the following well-known identities \cite{chowluni}:
\begin{align}
	\label{deltaS}
	\SS + \triangle f &= n\lambda,\\
	\label{rcandf}
	\Rc(\nabla{f})&=\frac{1}{2}\nabla{\SS},\\ 
	\label{nablafandS}
	\SS+|\nabla f|^2-2\lambda f &= \text{constant},\\
	\label{lapS}
	\triangle\SS+2|\text{Rc}|^2 &= \left\langle{\nabla f,\nabla \SS}\right\rangle+2\lambda \SS.
\end{align}

In the presence of a complex structure, there are further observations. When a manifold $M$ is of an even dimension, an almost complex structure is defined to be a smooth section $J$ of the bundle of endormorphisms $\text{End}(TM)$ such that $J^2=-\id.$ $J$ is said to be integrable if it is induced from an atlas of complex charts with holomorphic transition functions. $(M, g, J)$ is called an almost Hermitian manifold and $g$ a Hermitian metric if $g(JX, JY)=g(X, Y).$ The fundamental $2$-form or K\"{a}hler form is given by $\omega_g(X, Y)= g(X, JY).$ $(M, g, J)$ is called almost K\"{a}hler if $d\omega=0$. When $J$ is integrable, one upgrades an almost Hermitian to Hermitian and almost K\"{a}hler to K\"{a}hler. For a Riemannian manifold to be K\"{a}hler, the following is well-known.
\begin{proposition}\cite[Proposition 3.1.9]{BGbookSasakian08}
	\label{charKahler}
	Let $(M, g, J)$ be an almost Hermitian real manifold. $(M, g, J)$ is K\"{a}hler iff $\nabla J=0$
\end{proposition} 

\begin{definition}
	$(M, g, J, f)$ is a K\"{a}hler GRS if $(M, g, f)$ is a GRS and $(M, g, J)$ is a K\"{a}hler manifold.
\end{definition}
	
	It is crucial to observe that, on a K\"{a}hler manifold $(M, g, J)$, $\Rc$ is $J$-invariant. Thus, the following is well-known \cite{fik03, koi90, caohd09}.

\begin{lemma}
	\label{killing1}
	Let $(M, g, J)$ be a K\"{a}hler manifold and $f:M \mapsto \mathbb{R}$ such that $\He f$ is $J$-invariant. Then, we have the following:
	\begin{enumerate} [label=(\roman*)]
		\item $J(\nabla f)$ is a Killing vector field.
		\item $\nabla f$ is an infinitesimal automorphism of $J$. 
	\end{enumerate}
\end{lemma}
\subsection{Foliation} 
\label{foliation}

In this subsection, we recall the concept of a foliation and related properties. The references are \cite{tonFol97, BGbookSasakian08}. A $p$-dimensional foliation $\mathcal{F}$ of an $n$-dimensional manifold $M$ refers to the partition of $M$ into a union of disjoint $p$-dimensional immersed submanifolds $\{L_\alpha\}_{\alpha\in A}$, called leaves, with the following property. Every point is on a chart $U$ with coordinates $(x_1,... x_p;~ y_1,...y_{q}),~ p+q=n,$ such that, for each $L_\alpha$, a connected component of $U\cap L_\alpha$ is described by the equations \[y_1=\text{constant}, ..., y_q=\text{constant}.\]
Consequently, it is called a foliated coordinate chart and let $E$ denote the sub-bundle of $TM$ tangential to $\mathcal{F}$. A vector field $X$ is said to be foliate with respect to foliation $\mathcal{F}$ if for every section $Y\in E$, $\lie_X Y$ is also tangential.

A Riemannian metric $g$ induces an orthogonal decomposition $TM= E\oplus E^\perp$. A section of $E^\perp$ is said to be horizontal. 
\begin{definition}
	\label{bundlelike}
	A Riemannian metric is said to be bundle-like with respect to a foliation $\mathcal{F}$ if for any foliate horizontal vector fields $X, Y$ and a tangential $V$, $V g(X, Y)=0.$ In that case, the foliation is said to be Riemannian. 
\end{definition}
\begin{remark}
	Locally a foliation looks like a submersion and a Riemannian foliation corresponds to a Riemannian submersion. 
\end{remark}

The following will be important to our investigation.
\begin{proposition}\cite[Propositions 2.6.7 and 2.6.9]{BGbookSasakian08} \label{1dimKilling}
	A $1$-dimensional foliation induced by a Killing vector field is Riemannian. Also, a $1$-dimensional Riemannian foliation whose orbits are geodesics is isometric; that is, the flow generated by the associated vector field is an isometry.  
\end{proposition}


For a Riemannian foliation, the toolkit originally developed by B. O'Neill to study submersion  \cite{ONeill66} will play a crucial role. We let $\pi$ and $\pi^\perp$ be the projections from $TM$ onto $E$ and $E^\perp$ accordingly. Let $\nabla^\perp$ denote the following induced connection, for $Y$ a smooth section of $E^\perp$:
\begin{equation*}
	\nabla^\perp_X Y=\begin{cases}
		\pi^\perp (\nabla_X Y) \text{ if $X$ is a smooth section of $E^\perp$}\\
		\pi^\perp [X, Y] \text{ if $X$ is a smooth section of $E$}.
	\end{cases}
\end{equation*}
It is verified that $\nabla^\perp$ is the unique Levi-Civita connection with respect to the transverse metric $g^\perp:= g_{\mid_{E^\perp}}$ \cite[Exercise 2.4]{BGbookSasakian08}. Furthermore, we make the simplifying assumption that each leaf of $\mathcal{F}$ is totally geodesic. Let $U$ be a smooth section of $E$ and $X, Y$ be ones of $E^\perp$. In this setup, the connection is encoded by tensor $A$ as follows:
\begin{align*}
	A_X U &:=\pi^\perp (\nabla_X U),\\
	A_X Y &:= \pi (\nabla_X Y)=-A_Y X=\frac{1}{2}\pi([X, Y]).
\end{align*}

\subsection{Almost Contact and Contact Structures}
In this subsection, we give a brief introduction to almost contact geometry. The reference is \cite{BGbookSasakian08}. 

\begin{definition} 
	An odd dimensional manifold $P$ is called almost contact if there exists a triple $(\zeta, \eta, \Phi)$ where $\zeta$ is a vector field, $\eta$ is a $1$-form, $\Phi$ is a tensor field of type $(1, 1)$, and they satisfy, 	everywhere on $M$,  
	\[\eta(\zeta)=1 \text{  and } \Phi^2=-\id+\zeta\otimes \eta.\]
\end{definition}
\noindent It is immediate from the definition that $\zeta$ is nowhere vanishing and, by Frobenius theorem, $\zeta$ generates an $1$-dimensional foliation $\mathcal{F}$ ($\zeta$ is also called the Reeb vector field \cite{reeb52}). Additionally, $\Phi$ is non-degenerate on the transverse subspace and 
\[\Phi(\zeta)=0, ~~~ \eta \circ \Phi =0.\]
In the presence of a Riemannian metric $g$, $(P, \zeta, \eta, \Phi)$ is called an almost contact metric structure if $g$ is compatible with $\Phi$. That is,   
	\[g(\Phi(X), \Phi(Y))=g(X, Y)-\eta(X)\eta(Y). \]
It is naturally of great interest to consider the transverse geometry $g^\perp$, which is the restriction of $g$ to the subspace transverse to $\zeta$: 
\[g=g^\perp +\eta\otimes \eta. \]	


In our analysis, it is important to determine when leaves of $\mathcal{F}$ are geodesics. 
\begin{lemma} \label{almostgeodesic}
	 Let $(P, \zeta, \eta, \Phi)$ be an almost contact metric structure. Then the following are equivalent:
	 \begin{enumerate}[label=(\roman*)]
	 	\item Integral curves of $\zeta$ are geodesics.
	 	\item $\eta$ is invariant along the flow of $\zeta$.
	 	\item $(\nabla_\zeta \Phi)\zeta=0$.
	 	\item $d\eta(\zeta, \cdot)=0$.
	 \end{enumerate} 
\end{lemma}

\begin{proof} It follows from standard calculation; see \cite{BGbookSasakian08}.
	
	




\end{proof}

Furthermore, the tensor $A$ can be computed immediately \cite[Prop 2.5.14]{BGbookSasakian08}.

\begin{lemma} \label{computeA} Let $(P, \zeta, \eta, \Phi)$ be an almost contact metric structure such that $g$ is bundle-like with respect to the foliation generated by $\zeta$. Then, for horizontal vector fields $X$ and $Y$,
	\begin{align*}
		2g(A_X Y, \zeta) &=-2d\eta(X, Y)=g([{X}, {Y}], \zeta), \\
		g(A_X \zeta, Y) &= d\eta(X, Y).
	\end{align*}
\end{lemma}

The $\Phi$-sectional curvature of an almost contact manifold $(P, \zeta, \eta, \Phi)$ is defined on the horizontal sub-bundle (perpendicular to $\zeta$), for unit length $X$,
\[K_\Phi(X)=K(X, \Phi(X)).\]


\begin{definition}
	An almost contact metric structure $(P, \zeta, \eta, \Phi, g)$ is called contact if one further assumes $g(X, \Phi(Y))=d\eta(X, Y).$
\end{definition}
Additionally, a contact metric manifold $(P,\zeta, \eta, \Phi, g)$ is called K-contact if $\zeta$ is Killing; that is, $\lie_{\zeta}g=0.$ We are interested in certain $K$-contact structures which give a concrete bridge from almost contact (contact) to almost complex (complex).

\begin{definition}
	A Sasakian structure is a contact metric structure such that the cone $(P\times \mathbb{R}^+, r^2 g+ dr^2)$ is K\"{a}hler. Moreover, the K\"{a}hler form is given by $d(r^2\eta)$. 
\end{definition}

Next we recall a transformation defined in \cite{tanno69}.

\begin{definition}
	\label{Sasadeform}
	For $H, F\in \mathbb{R}^+$, a $(\pm, H, F)$-deformation is given by 
	\[ \zeta^\ast =\frac{\zeta}{H}, ~~~ \eta^\ast = H\eta, ~~~ \Phi^\ast=\pm\Phi,~~~ g^\ast=  F^2 g +(H^2-F^2)\eta\otimes \eta. \]
	A deformed contact/Sasakian metric structure is obtained via an $(\pm, H, F)$-deformation of a contact/Sasakian metric structure.  
\end{definition}
\begin{remark}
	\label{scalingdiff} 
	A $(+, H, F)$ transformation of a contact metric structure is contact iff $H= F^2$; this special case is called the \textit{transverse} homothety deformation \cite[Page 228]{BGbookSasakian08}.
	
\end{remark}
It is of interest to relate the curvature of a $(\pm, H, F)$ deformation with one of the original metric. First, there is a simple observation.

\begin{lemma} \label{basicFH}
	Let $(P, g^\ast, \zeta^\ast, \eta^\ast, \Phi^\ast)$ be an $(\pm, H, F)$ deformation of an almost contact metric structure. Then we have the following:
	\begin{enumerate}[label=(\roman*)]
		\item $(P, g^\ast, \zeta^\ast, \eta^\ast, \Phi^\ast)$ is an almost contact metric structure.
		\item  $g$ is bundle-like with respect to $\mathcal{F}$ if and only if $g^\ast$ is.
		\item $\lie_{\zeta} g=0 \iff \lie_{\zeta^\ast} g^\ast=0$.  
	\end{enumerate}	
\end{lemma}
\begin{proof}
	The proof is via straightforward verification. For example, 
	\begin{align*}
		g^\ast (\Phi^\ast X, \Phi^\ast Y) &= (F^2 g +(H^2-F^2)\eta\otimes \eta)(\Phi X, \Phi X)\\
		&= F^2 (g(X, Y)-\eta(X)\eta(Y))= g^\ast (X, Y)-\eta^\ast(X)\eta^\ast(Y). 
	\end{align*}
\end{proof}
\noindent For $g=g^\perp+ \eta\otimes \eta$,  
	\[{g^\ast}= H^2 \eta\otimes \eta+ F^2 g^\perp. \]
In case of a  K-contact structure, the curvature calculation is relatively simple.  
\begin{proposition}
	\label{Rcdeformed}
	Let $(P, \zeta,\eta, \Phi, g)$ be a K-contact structure and $(P, g^\ast, \zeta^\ast, \eta^\ast, \Phi^\ast)$ be its $(\pm, H, F)$ deformation. Then the curvatures are related by, for orthonormal horizontal vectors $X$ and $Y$
	\begin{align*}
		{K^\ast}(X, Y) &= \frac{1}{F^2}K^\perp(FX, FY)-3\frac{H^2}{F^4} g^\ast (X, \Phi^\ast Y)^2,\\
		{K^\ast}(X, \zeta^\ast) &= \frac{H^2}{F^4},\\
		{\Rc^\ast}(X, Y) &= \Rc^\perp (X, Y)-2\frac{H^2}{F^4}g^\ast(X, Y)\\\
		\Rc^\ast(\zeta^\ast, \zeta^\ast) &=\frac{H^2}{F^4} (\text{dim}(P)-1).
	\end{align*}
Here $K^\perp$ and $\Rc^\perp$ are the sectional and Ricci curvature of $(g^\perp, \nabla^\perp)$.  
\end{proposition}
\begin{proof}
	Since $\zeta$ is a Killing vector field, its generated foliation $\mathcal{F}$ is Riemannian by Proposition \ref{1dimKilling}. By Lemma \ref{basicFH}, $(M, g^\ast, \zeta^\ast, \eta^\ast, \Phi^\ast)$ is an almost contact metric structure, $g^\ast$ is bundle-like with respect to $\mathcal{F}$, and  $\lie_{\zeta} g=0= \lie_{\zeta^\ast} g^\ast$. Therefore, by Lemma \ref{almostgeodesic}, orbits of $\mathcal{F}$ are geodesics with respect to either $g$ or $g^\ast$.  

	For a Riemannian totally geodesic foliation, the curvature can be computed via tensor $A$; see \cite[Chapter 9]{besse} or \cite[Theorem 2.5.16]{BGbookSasakian08}. In our case, the computation of $A$ is given by Lemma \ref{computeA}. 
\end{proof}

\subsection{Rectifiable, Transnormal, and Isoparametric Functions}
\label{isoparasection}
In this subsection, for a complete connected Riemannian manifold $(M, g)$, we consider a smooth function $f: M^n \rightarrow \mathbb{R}$. First, we recall the definition of rectifiable which appears in \cite{PW09grsym, CM16gradient}. 
\begin{definition}
	$f$ is rectifiable if $|\nabla f|$ is constant along each connected component of level sets of $f$.
\end{definition}
A connected component of a regular level set is called a regular connected component. 
\begin{lemma} \label{equivalentconditions}
	The following are equivalent:
\begin{enumerate}[label=(\roman*)]
	\item f is rectifiable.
	\item Integral curves of $\nabla f$ are geodesics after reparametrization.
	\item The gradient of $f$ is an eigenvector of its Hessian.
\end{enumerate}
\end{lemma}
\begin{proof}
Integral curves of $\nabla f$ are reparametrized geodesics if and only if, when $\nabla f\neq 0$,
	\begin{align*}
		0 &=\frac{1}{|\nabla f|} \nabla_{\nabla f}\frac{\nabla f}{|\nabla f|}\\ &=\frac{1}{|\nabla f|^3}\big(|\nabla f| \nabla_{\nabla f}\nabla f-(\nabla_{\nabla f}|\nabla f|)\nabla f\big).
	\end{align*}
	Thus, it is equivalent to that, for any unit vector field $E_i\perp \nabla f$, 
	\[0= g(\nabla_{\nabla f}\nabla f, E_i)=\He f(\nabla f, E_i)= \frac{1}{2}\nabla_{E_i}|\nabla f|^2.\]
	Equivalently, $|\nabla f|^2$ is constant on each regular connected component. The equation above also is equivalent to that $\nabla f$ is an eigenvector of $\He f$.
\end{proof}
\begin{remark}
	Consequently, it justifies the terminology that a rectifiable function is also called one with a geodesic gradient \cite{derd21, dp20geodesicholomorphicgradient}. 
\end{remark}
A priori, $|\nabla f|$ might vary between different connected components. Thus, it is useful to have a global condition.  

\begin{definition}
	$f$ is called transnormal if there is a continuous function $b: \mathbb{R}\mapsto \mathbb{R}$ such that $|\nabla f|^2= b\circ f.$	Furthermore, a transnormal function is called \textit{isoparametric} if there is a continuous function $a:\mathbb{R}\mapsto \mathbb{R}$ such that $\Delta f= a\circ f.$ 
\end{definition}

	Generally speaking, the former condition corresponds to equidistant level sets while the latter implies that each regular one has constant mean curvature.  
Also, this subject is closely related to the theory of a transnormal system and polar foliations \cite{miyaoka13}. 	
It is immediate that if $f$ is rectifiable and each level set is connected then $f$ is transnormal. 


\begin{remark}
	\label{fsegment}
A geodesic segment $\gamma:[\alpha, \beta]\mapsto M$ is called an $f$-segment if 
$\gamma'(t)=\frac{\nabla f}{|\nabla f|}$ whenever $|\nabla f|\neq 0$. At a point p such that $\nabla f(p)\neq 0$, it is possible to reparametrize an $f$-segment $\gamma$ via translation such that $f\circ\gamma$ is strictly increasing, so that $t$ and $f$ determine each other, in a neighborhood of $p$. If $f(p)$ is a local minimum or maximum, one can only reparametrize to obtain such a property in an one-sided neighborhood. Also, with this parametrization, $\frac{\partial f}{\partial t}=|\nabla f|$.  
\end{remark}

For $c\in f(M)$, denote $M_c=f^{-1}(c)$. 

\begin{lemma} 
	\label{isopara1} If $f$ is transnormal and $[\alpha, \beta]\subset f(M)$ contains no critical value of $f$ then for any $x\in M_\alpha$, $y\in M_\beta$, we have
	\begin{enumerate} [label=(\roman*)]
		\item $d(x, M_\beta)= d(M_\alpha, y)= \int_{\alpha}^\beta \frac{dc}{\sqrt{b(c)}}$;
		\item the integral curves of $\nabla f$ after reparametrization are $f$-segments;
		\item the $f$-segments are the shortest curves among all curves connecting $M_\alpha$ and $M_\beta$.
	\end{enumerate}
\end{lemma}
\begin{proof}
	See \cite[Lemma 1]{wang87iso}.
\end{proof}

Obviously, there is a local version when $f$ is rectifiable. Thus, a Riemannian manifold with a rectifiable function $f$ is locally foliated by equidistant hypersurfaces and one can write the metric in specific way. Following \cite[Section 2]{EW00ivp}, let $P$ be a differentiable manifold corresponding to a regular connected component. There is a local diffeomorphism $\phi: I\times P\mapsto M$ such that the metric can be written as  $\phi^\ast g=dt^2+ g_t.$ Here $t$ is a parametrization of $f(M)$ with unit tangent vector (see Remark \ref{fsegment}) and $g_t$ is an one-parameter family of metrics $P$ such that, for $\phi_c: \{c\}\times P\mapsto M_c$ the restriction of $\phi$ to a slice, $\phi_c^\ast g_{\mid M_c}= g_c.$

For a level set $M_c$, we denote the normal exponential map $\Pi_c: T^\perp M_c \mapsto M.$ At a regular value, $\Pi_c$ induces a diffeomorphism between nearby regular connected components of distance $\epsilon$: $\Pi_c^{\epsilon}: M_c \mapsto M_{c+\epsilon}$, via normal geodesics of length $\epsilon$ in the direction of $\partial_t$. 

\begin{lemma}
	\label{identification}
	Let $I$ be a continuous open interval with regular values of $f$. For $(c-\epsilon, c+\epsilon)\subset I$, then $\Pi_c^\epsilon$ is just the identification by the diffeomorphism $\phi$. 
\end{lemma}
\begin{proof}
	For $p$ a point in $P$ and $c\in I$ such that $\phi(p, c)=q\in M_c$. Let $\gamma(t)$ be the curve given by, for $t\in (c-\epsilon, c+\epsilon)$, $ \gamma(t)=\phi (p, t).$	It is readily verified that $\ddt \gamma= \phi_{\ast} \partial_t.$ Thus,  $\gamma$ is a geodesic segment and by the uniqueness of a geodesic given initial conditions, we conclude that
	\[\Pi^\epsilon_c (q)= \gamma(\epsilon)=\phi(p, \epsilon). \]
	
\end{proof}



A soliton structure comes with a potential function and Lemma \ref{equivalentconditions} is refined as follows, which is also implicit in \cite{PW09grsym}.
\begin{proposition}
	\label{rectequivalent}
	Let $(M, g, f)$ be a GRS with a non-constant $f$. The following are equivalent:
	\begin{enumerate} [label=(\roman*)]
		\item f is rectifiable,
		\item whenever $\nabla f\neq 0$, $\nabla f$ is an eigenvector of $\Rc$,
		\item $\nabla f \parallel \nabla \SS$ everywhere.
	\end{enumerate}
	Furthermore, $f$ is transnormal iff $f$ is isoparametric.  
\end{proposition}

\begin{proof} $f$ is rectifiable iff $|\nabla f|$ is constant on each regular connected component and, by equation (\ref{nablafandS}), so is $\SS$. By Sard's theorem, the set of regular values is dense. Thus, by continuity, it is equivalent to $\nabla f \parallel \nabla \SS$ everywhere. That establishes $(i)\iff (iii)$. Finally, $(ii) \iff (iii)$ follows from equation (\ref{rcandf}).
	
	
	
	
	
	
	
	
	
\end{proof}
	
	
	
	
When $f$ is rectifiable, $b: f(M)\mapsto \RR$ can be defined locally and is smooth at a regular value. At a critical one, it is only smooth in the sense of one-sided limits. 
Furthermore, the local foliation of equidistant hypersurfaces allows one to rewrite the soliton equation as follows \cite{DW11coho}. For $N=\partial_t$, the shape operator is given as $L X := \nabla_X N.$ Denoting the ordinary derivative $\frac{d}{dt}$ by $'$, it follows that
\begin{align}
	\label{Levolve}
	g'&= 2 g\circ L, ~~~~~~	\nabla_{N} L= L'.
\end{align}
\noindent 
Due to Gauss, Codazzi, and Riccati equations, the Ricci curvature of $(M, g)$ is determined by that of $(M_t, g_t)$ and the shape operator. That is,  
\begin{align}
	\Rc(X, Y) &=\Rc_t(X, Y)-\tr(L_t)g_t(LX, Y)-g_t({L'}(X), Y), \nonumber\\
	\label{Rccomp}
	\Rc(X, N) &=-\nabla_X \tr(L_t)-g_t(\delta L, X),\\ 
	\Rc(N, N) &= -\text{tr}({L'})-\text{tr}(L^2).\nonumber
\end{align}
\noindent
Here $\Rc_t$ denotes the Ricci curvature of $(M_t, g_t)$, $\tr{T}=\tr_{g_t}T_t$, and $\delta$ the co-differential. Next, we recall $\He f(X, Y)=g(\nabla_X {\nabla f}, Y)$ and $\nabla f= \frac{df}{dt} N=f' N$. Consequently, 
the gradient Ricci soliton equation $\Rc+\He{f}=\lambda g$ is reduced to:
\begin{align}
	0 &=-(\delta L)-\nabla \tr{L},\nonumber \\
	\label{reducedsolitionsystem}
	\lambda &=-\text{tr}({L'})-\text{tr}(L^2)+f'',\\
	\lambda g(X, Y) &=\Rc_t(X, Y)-(\tr{L}) g(L X, Y)-g({L'}(X), Y) +f' g(LX, Y). \nonumber
\end{align}

\section{Deformed Sasakian}
\label{defornedsasakian}
In this section, we give a characterization for an almost contact metric manifold to be a deformed Sasakian structure, proving Theorem \ref{chardeformed}. Consequently, it is possible to detect such a structure on a real hypersurface of a K\"{a}hler manifold via examining its shape operator as in Theorem \ref{hyperdeformedSasa}. Thus, they generalize classical results, \cite[Theorem 7.3.16 and Theorem 7.3.18]{BGbookSasakian08}, and might be of independent interest. 

Let $(P, \zeta, \eta, \Phi, g)$ be an almost contact metric manifold and $\mathcal{F}$ be the foliation generated by $\zeta$. The following is well-known.  
\begin{regtheorem}\cite[Theorem 7.3.19]{BGbookSasakian08}
	\label{classicalSasa}
	$(P, \zeta, \eta, \Phi, g)$ is Sasakian if and only if, for every vector fields $X$ and $Y$,
	\begin{equation}
		\big(\nabla_X \Phi\big) Y =  g(X, Y)\zeta- \eta(Y)X. \end{equation}
\end{regtheorem} 

Towards a generalization, we observe the following. 

\begin{lemma}
	\label{killing}
If	$\big(\nabla_X \Phi\big) Y =  bg(X, Y)\zeta- b\eta(Y)X$ for a real number $b$ and any vector fields $X$ and $Y$ then $g$ is bundle-like with respect to $\mathcal{F}$ and $\zeta$ is Killing. 
\end{lemma}
\begin{proof}
	The assumption implies that, interchanging $X$ and $Y$,
	\begin{align*}
		b (\eta(X)Y-\eta(Y)X) &=\big(\nabla_X \Phi\big) Y-	\big(\nabla_Y \Phi\big) X= \nabla_X (\Phi Y)-\nabla_Y (\Phi X)-\Phi([X, Y])
\end{align*}
Let $X=\zeta$ and $Y$ be a foliate horizontal vector field then $[X, Y]$ is tangential and, thus,
\[\nabla_Y (\Phi X)=\Phi([X, Y])=0.\] 
Consequently, $\nabla_\zeta (\Phi Y)=bY.$ Applying the assumption again yields
\begin{align*} 0 &= b g(\zeta, \Phi Y)\zeta-\eta(\Phi(Y))b\zeta =\nabla_\zeta \Phi(\Phi Y)\\
	&=\nabla_\zeta \Phi^2(Y)-\Phi(\nabla_\zeta (\Phi Y))= \nabla_\zeta (-Y) - b\Phi (Y). 
\end{align*} 
Thus, for foliate horizontal vector fields $Y$ and $Z$,
\begin{align*}
\zeta g(Y, Z) &= g(\nabla_\zeta Y, Z)+ g(Y, \nabla_\zeta Z)= -b g(\Phi Y, Z)-b g(Y, \Phi Z)=0.
\end{align*}
The last equality is due to the compatibility of $g$ and $\Phi$. Therefore, $g$ is bundle-like by Definition \ref{bundlelike} and $\mathcal{F}$ is a Riemannian foliation. Furthermore, Lemma \ref{almostgeodesic} is also applicable since $\nabla_{\zeta}\Phi (\zeta)=b\zeta-b\zeta=0.$ Thus $\mathcal{F}$ is a Riemannian foliation whose orbits are geodesics. By Proposition \ref{1dimKilling}, $\zeta$ is a Killing vector field.

\end{proof}

\begin{lemma}
	\label{nablazeta2} For a horizontal vector field $X$, the following are equivalent
	\begin{enumerate} [label=(\roman*)]
		\item $(\nabla_\zeta \Phi) X= 0$
		\item $[\zeta, \Phi(X)]-\Phi([\zeta, X])=\Phi(A_X\zeta)-A_{\Phi X} \zeta$.
	\end{enumerate}
\end{lemma}
\begin{proof}
	We compute, using the notation from Section \ref{preliminary},
	\begin{align*}
		(\nabla_\zeta \Phi) (X) &= \nabla_\zeta (\Phi X)-\Phi(\nabla_\zeta X)\\
		&= [\zeta, \Phi X]-\Phi([\zeta, X])-\Phi(\nabla_X\zeta)+\nabla_{\Phi X} \zeta,\\
		&= [\zeta, \Phi X]-\Phi([\zeta, X])-\Phi(A_X\zeta)+A_{\Phi X} \zeta.
	\end{align*}
\end{proof}

Moreover, since $\mathcal{F}$ is Riemannian, by Lemma \ref{computeA}, the tensor $A$ can be computed as, for horizontal vector fields $X$ and $Y$,
\begin{align*}
	2A_X Y &=-2d\eta(X, Y)\zeta=g([{X}, {Y}], \zeta)\zeta \\
	g(A_X \zeta, Y) &= d\eta(X, Y).
\end{align*}	

\begin{lemma}
	\label{nablaXPhi1}
	For horizontal vector fields $X, Y$, the following are equivalent
	\begin{enumerate} [label=(\roman*)]
		\item $(\nabla_X \Phi)Y= b g(X, Y)\zeta$
		\item $A_X (\Phi Y)= b g^\perp (X, Y)\zeta$ and $\nabla_X^\perp (\Phi Y)=\Phi (\nabla_X^\perp Y)$
	\end{enumerate}
\end{lemma}

\begin{proof}
	We compute, 
	\begin{align*}
		(\nabla_X \Phi) (Y) &= (\nabla_X (\Phi Y)-\Phi(\nabla_X Y)\\
		&=A_X (\Phi Y)+\pi^\perp (\nabla_X (\Phi Y))-\Phi(A_X Y+ \pi^\perp (\nabla_X Y))\\
		&=A_X (\Phi Y)+\nabla^\perp_X (\Phi Y))-\Phi(\nabla^\perp_X Y).
	\end{align*}
\end{proof}
	\begin{lemma}
		\label{nablaXPhi2}
		For horizontal vector fields $X, Y$, the following are equivalent
		\begin{enumerate} [label=(\roman*)]
			\item $(\nabla_X \Phi)\zeta= -bX$
			\item $A_X \zeta= -b\Phi X$. 
		\end{enumerate}
	\end{lemma}
	
	\begin{proof}
		We compute
		\begin{align*}
			(\nabla_X \Phi) (\zeta) &= (\nabla_X (\Phi(\zeta))-\Phi(\nabla_X \zeta)\\
			&=-\Phi(A_X\zeta).
		\end{align*}
	\end{proof}

Thus, it is possible to track how the covariant derivative of $\Phi$ changes under a deformation. 
\begin{regtheorem}
	\label{varcon}
	Let  $(P, \zeta, \eta, \Phi, g)$ be an almost contact metric structure such that, for every vector fields $X$ and $Y$,
	\begin{equation*}
		\nabla_X \Phi (Y) =  bg(X, Y)\zeta- \eta(Y)bX.
	\end{equation*} 
Let $(P, \zeta', \eta', \Phi', g')$ be an $(\pm, 1, c)$-deformation then
	\begin{equation*}
		\nabla'_X \Phi' (Y) =  \pm \frac{b}{c^2} \big(g'(X, Y)\zeta'- \eta'(Y)X\big).
	\end{equation*} 
\end{regtheorem}

\begin{proof}
 Let $\mathcal{F}$ be the foliation generated by $\zeta$. By Lemma \ref{killing}, $\mathcal{F}$ is Riemannian and $\zeta$ is a Killing vector field. We write the metric as 
	\begin{equation*}
		g= g^\perp +\eta\otimes \eta.
	\end{equation*}
	By Lemmas \ref{nablaXPhi1} and \ref{nablaXPhi2}, for horizontal vector fields $X, Y$,
	\begin{align}
		A_X\Phi(Y) &= b g(X, Y)\zeta,\nonumber\\
		\label{auxideformed}
		\nabla^\perp_X \Phi(Y) &= \Phi(\nabla_X^\perp Y),\\
		A_X\zeta &= -b\Phi(X)\nonumber. 
	\end{align}
We will assume $(P, \zeta', \eta', \Phi', g')$ be an $(1, c)$-deformation as the minus case can be done similarly. Thus, $\eta'= \eta$, $\zeta'=\zeta$, $\Phi'=\Phi$ and
	\[ g'= c^2 g+(1-c^2)\eta\otimes \eta= c^2 g^\perp + \eta\otimes \eta. \]
By Lemma \ref{basicFH}, $(P, g', \eta', \zeta', \Phi')$ is an almost contact metric structure, $\zeta=\zeta'$ is a Killing unit vector field with respect to $g'$, and $g'$ is bundle-like with respect to $\mathcal{F}$. By Lemma \ref{computeA} and equation (\ref{auxideformed}), we have
	\begin{align*}
		A'_X (\Phi' Y) &=-d\eta'(X, \Phi'(Y))\zeta'= -d\eta(X, \Phi(Y))\zeta= A_X\Phi(Y)\\
		&= bg^\perp (X, Y)\zeta= \frac{b}{c^2} g'(X, Y)\zeta'. 
	\end{align*}
	Similarly, 
	\begin{align*}
		g'(A'_X \zeta', Y)  &=d\eta'(X, Y)=d\eta(X, Y)=g(A_X\zeta, Y)\\
		&= -g(b\Phi(X), Y)=-\frac{b}{c^2}g'(\Phi(X), Y). 
	\end{align*}
Thus, $A'_X\zeta= -\frac{b}{c^2}\Phi(X)$. As the Levi-Civita connection is scaling invariant, $\nabla^{c^2 g^\perp}= \nabla^{g^\perp}=\nabla^\perp$. By Lemma \ref{nablazeta2}, $\nabla_\zeta \Phi(X)=0$ if and only if
\begin{align*}
	[\zeta, \Phi(X)]-\Phi([\zeta, X]) &=\Phi(A_X\zeta)-A_{\Phi X} \zeta=\Phi(-b\Phi(X))+b\Phi(\Phi(X))=0.
\end{align*}
As the left-hand side is independent of the metric, $\nabla'_{\zeta'} \Phi'(X)=(\nabla_\zeta \Phi)(X)=0$. Together with Lemmas \ref{nablaXPhi1} and \ref{nablaXPhi2}, we obtain, for any vector fields $V$ and $W$,
\[\nabla'_V \Phi' (W) =  \frac{b}{c^2} \Big( g'(V, W)\zeta'- \eta'(W)V\Big).\]

\end{proof}

\begin{regcorollary}
	\label{varconHF}
	Let  $(P, \zeta, \eta, \Phi, g)$ be a Sasakian manifold and $(P, \zeta', \eta', \Phi', g')$ be a $(\pm, H, F)$-deformation. Then
	\begin{equation*}
		\nabla'_X \Phi' (Y) = \frac{\pm H}{F^2} \big(g'(X, Y)\zeta'- \eta'(Y)X\big).
	\end{equation*} 
\end{regcorollary}

\begin{proof}
	By Remark \ref{scalingdiff}, a $(+,H, \sqrt{H})$-deformation is Sasakian. Thus, applying Theorem \ref{varcon} for $b=1, c= \frac{F}{\sqrt{H}}$ in combination with Theorem \ref{classicalSasa} yields the result.  
\end{proof}
We are ready to give the proof of Theorem \ref{chardeformed}. 
\begin{proof}[Proof of Theorem \ref{chardeformed}]
		For part (i), without loss of generality, we assume $b>0$ (if $b<0$ then we consider $(P, \zeta, \eta, -\Phi$). One direction follows from Corollary \ref{varconHF}. The other is deduced by applying Theorem \ref{varcon} for $c=\sqrt{b}.$

	For part (ii), $b=0$ means $\Phi$ is parallel since, for every vector fields $X$ and $Y$ 
	\[\nabla_X \Phi(Y)=0.\]
	By Lemmas \ref{almostgeodesic} and \ref{killing}, $\zeta$ is a Killing and geodesic vector field. By Lemma \ref{nablaXPhi2}, for any horizontal vector field $X$,
	\[\nabla_X \zeta= A_X \zeta=0.\]
	Thus, $\zeta$ is a unit-length parallel vector field and, accordingly, $P$ is locally a Riemannian product of an interval and an integral manifold of the parallel distribution $\zeta^\perp$. If, moreover, $g$ is complete, then, by the de Rham decomposition theorem \cite[Chapter 10]{besse}, the universal cover of $P$ is a Riemannian product of a line and a simply connected integral manifold of $\zeta^\perp$. The factor perpendicular to $\zeta$ has $\Phi$ as an almost complex structure. Its compatibility with $g$ makes the induced metric almost Hermitian. Since $\Phi$ is parallel, the metric is {K}\"{a}hler by Proposition \ref{charKahler}. Conversely, if $P$ is locally such a product and $\Phi$ is the pull-back of the complex structure of the K\"{a}hler factor, then $\nabla\Phi=0$. 
\end{proof}

\begin{proof}[Proof of Theorem \ref{hyperdeformedSasa}] Let $\nabla$ and $\nabla^P$ denote the Levi-Civita connection on $(M, g)$ and $P$ with the induced metric, respectively. We have
	\[\nabla_X Y= (\nabla^P_X Y)- g(LX, Y) V.\]
	Since $J$ is parallel with respect to $\nabla$, $\nabla_X (JY) = J (\nabla_X Y)$. Thus,  
	\begin{align*} 
		\nabla^P_X (\Phi Y)- g(LX, \Phi Y)V+X\eta(Y)V+\eta(Y)LX &= \Phi(\nabla^P_X Y)+ \eta(\nabla_X Y)V+g(LX, Y)\zeta.
	\end{align*}
	Consequently, equating the normal and tangent components give
	\begin{align*}
		(\nabla^P_X \Phi)Y &=- \eta(Y)LX+ g(LX, Y)\zeta,\\
		-g(LX, \Phi Y) +X\eta(Y) &= \eta(\nabla_X Y)
	\end{align*}
	For the first equation, substituting $L=\alpha \id+\beta \zeta\otimes \eta$ yields
	\begin{align*}
		(\nabla^P_X \Phi)Y &= -\eta(Y)\Big(\alpha X+ \beta\zeta\eta(X)\Big)+ g\Big(\alpha X+ \beta \zeta\eta(X), Y\Big)\zeta\\
		&=-\alpha \eta(Y)X+ \alpha g(X, Y)\zeta.
		\end{align*}
	Similarly, for the second equation, we have
	\begin{align*}
		\eta(\nabla_X Y)&= -g(\alpha X+ \beta \zeta\eta(X), \Phi Y) +X\eta(Y),\\
	\iff	g(\nabla_X \zeta, Y) &= \alpha g(X, \Phi (Y))=\alpha g(\Phi(X),-Y),\\
	\iff	\nabla_X \zeta &= -\alpha\Phi(X). 
	\end{align*}
By Lemmas \ref{killing} and \ref{nablaXPhi2}, the second equation is implied by the first. The result then follows from Theorem \ref{chardeformed}.

Conversely, suppose that $(\nabla^P_X \Phi)Y=b\big(g(X, Y)\zeta-\eta(Y)X\big)$ for a constant $b$; by Theorem \ref{chardeformed}, this holds for a deformed Sasakian structure (with $b\neq 0$) and for a structure with $\nabla\Phi=0$ (with $b=0$). Comparing with the first equation above at $Y=\zeta$ gives $LX-g(LX, \zeta)\zeta=b\big(X-\eta(X)\zeta\big)$ for all $X$. Hence $L\zeta$ is a multiple of $\zeta$ and, by the symmetry of $L$, $LX=bX$ for $X\perp\zeta$; that is, $L=b\,\text{Id}+\beta\zeta\otimes\eta$ for a function $\beta$.  
\end{proof}

\section{K\"{a}hler Solitons and Contact Structures}
\label{kahlerandcontac}
In this section, we'll study a K\"{a}hler GRS $(M, g, J, f)$ under the assumption that each regular connected component level set of $f$ is a deformed contact structure. Thus, a proof of Theorem \ref{HDcontact} will be provided.  For $c\in f(M)$, we recall the construction (\ref{standardconstruction}), for a regular value $c$,
\[V:=\frac{\nabla f}{|\nabla f|},~~ g_c:= g_{\mid M_c}, ~~\zeta_c:=-J(V),~~\eta_c(\cdot):= g(\cdot, \zeta_c), ~~ \Phi_c (\cdot):= -\eta_c(\cdot) V+ J(\cdot).  \]
Together, $(M_c, \zeta_c, \eta_c, \Phi_c, g_c)$ is an almost contact metric structure. Additionally, the shape operator is given by 
\[ g(L X, Y) =g(\nabla_X V, Y)=g\big(\nabla_X \frac{\nabla f}{|\nabla f|}, Y\big)=\frac{|\nabla f|\He f(X,Y)-X|\nabla f| g(\nabla f, Y) }{|\nabla f|^2}.\]
Thus,
\begin{align}\label{LandHess}
	\He f(X,Y)=|\nabla f|g(L X, Y)+\frac{X|\nabla f|}{|\nabla f|}g(\nabla f, Y).\end{align}
We observe the following.
\begin{lemma}
	\label{HDftransnormal}
	If $(M_c, \zeta_c, \eta_c, \Phi_c, g_c)$ is a deformed connected contact metric structure, then it is a deformed $K$-contact structure and $|\nabla f|$ is constant along $M_c$. 
\end{lemma}
\begin{proof}
	For convenience, we drop the $c$ dependence. By Lemma \ref{killing1}, $W:=J(\nabla f)= -|\nabla f|\zeta$ is a Killing vector field. It suffices to show $|\nabla f|$ is constant along $M_c$. First, we observe that, since $\He f$ is $J$-invariant:
	\begin{align*}
		\nabla_W |\nabla f|^2 &= 2g(\nabla_W \nabla f, \nabla f)=2 \He f(W, \nabla f)=0. 
	\end{align*}
	Next, the deformed contact metric structure implies that
	\[d\eta(\zeta, \cdot)=0.\]
	By Lemma \ref{almostgeodesic}, the integral curves of $\zeta$ are geodesics. Therefore, for any $X$ tangential to $M_c$,
	\begin{align*}
		0=g(\nabla_{\frac{W}{|W|}}\frac{W}{|W|}, X) &=\frac{|W|g(\nabla_W W, X)-W|\nabla f| g(W, X) }{|\nabla f|^3}.
	\end{align*}
	Since $W |\nabla f|^2=0$ we have $g(\nabla_W W, X)=0$. Since $W$ is a Killing vector field,
	\[0=g(\nabla_X W, W)=\frac{1}{2} X|W|^2.\]
	As $|W|=|\nabla f|$ the result follows. 	
	 
\end{proof}
If $f$ is rectifiable, for horizontal vector fields $X, Y$,
\begin{align}
	2d\eta_c (X, Y) &=(\nabla_X \zeta_c, Y)-(\nabla_Y\zeta_c, X), \nonumber\\
	&=\frac{1}{|\nabla f|}\Big((\nabla_X(-J\nabla f), Y)-(\nabla_Y (-J\nabla f), X)\Big),\nonumber\\
	\label{detaandhess}
	&= \frac{2}{|\nabla f|}\He f(X, JY).
\end{align} 
Here we also use that $\He f$ is compatible with $J$.

\begin{lemma}
	\label{almostKahler}
	Let $(P, \zeta, \eta, \Phi, \eta\otimes \eta+g^\perp)$ be an almost contact metric structure with $d\eta=\sigma\, g^\perp(\cdot, \Phi(\cdot))$ for a constant $\sigma\in\{1, -1\}$ (for $\sigma=1$, a contact metric structure). Suppose the metric and the almost complex structure, on $I\times P$,
	\begin{align*}
		g &=dt^2+ H^2(t)\eta\otimes \eta+ F^2(t) g^\perp,\\
		J &= \partial_t \otimes H\eta - \frac{1}{H}\zeta\otimes dt +\Phi,
	\end{align*}
	are almost K\"{a}hler. Then,
		\[ FF'={\sigma}H.\]
\end{lemma}
\begin{proof}
	Let $\omega^\perp:=g^\perp(\cdot, \Phi(\cdot))$, so that $d\eta=\sigma\,\omega^\perp$. Recall $\omega_g(X, Y)= g(X, JY)$, the K\"{a}hler form becomes, 
	\begin{align*}
		\omega_g &= 2dt\wedge H\eta+ F^2 \omega^\perp.
	\end{align*}
	Thus, $d\omega_g = -2\sigma H dt\wedge \omega^\perp +2FF' dt\wedge\omega^\perp$ and one deduces
		\[FF'={\sigma}H.\]
\end{proof}


 We are ready to give the main proof of this section. 
 
\begin{proof}[Proof of Theorem \ref{HDcontact}]
By Lemma \ref{HDftransnormal}, $f$ is rectifiable and each level set is endowed with a deformed $K$-contact structure. Let $P$ be a differentiable manifold corresponding to a regular connected component of $f$. As in Subsection \ref{isoparasection}, there is a local diffeomorphism $\phi: I\times P\mapsto M$ such that
\[\phi^\ast(g)= dt^2+ g_t.\] 
for $(P, g_t)$ an one-parameter family of Riemannian metrics which are equal to the pullback of induced metric on nearby connected components. We observe, since $f$ is invariant by the Killing vector field $J(\nabla f)$
\begin{align*}
	0 &= \lie_{J\nabla f}\nabla f=-\lie_{\nabla f}J(\nabla f). 
\end{align*}
Therefore, $J(\nabla f)=-|\nabla f|_t\zeta_t$ is invariant by the flow generated by $\nabla f$. Thus, by Lemma \ref{identification}, $f'\zeta_t$ is identified with a fixed vector $\zeta$ on $P$. Furthermore, equations (\ref{Levolve}) and (\ref{LandHess}) and that $\nabla f$ an eigenvector of $\He f$ imply that the transverse subspace is also invariant as $t$ varies. Thus, we can fix a background one-form $\eta$ and write
\begin{align*} 
	g_t &= H^2(t)\eta\otimes \eta+ g^\perp_t,\\
	H &= f',
\end{align*} 
with $g^\perp_t$ being the restriction of $g_t$ to the transverse subspace. Since each is a deformed contact structure, restricted to the transverse subspace, $d(H\eta)(\cdot, \cdot)$ is a multiple of $g^\perp_t(J \cdot, \cdot)$. By equation (\ref{detaandhess}), $g_t^\perp$ is a multiple of $\He f$ restricted to the transverse subspace. Combining with equations (\ref{Levolve}) and (\ref{LandHess}), one deduces that $\ddt g_t^\perp$ is a multiple of $g_t^\perp$ with the scalar only depending on $t$. Thus, for a fixed transverse metric $g^\perp$,
\[g= dt^2+ H^2(t) \eta\otimes \eta+ F^2(t) g^\perp.\] 
In particular, nearby regular level sets are endowed with the deformation of the \textit{same} $K$-contact structure. Moreover, $d\eta$ is a non-zero constant multiple of $g^\perp(\cdot, \Phi(\cdot))$ on the transverse subspace, where $\Phi=J$ there; rescaling $g^\perp$ and $F$ by positive constants, we may assume that $d\eta=\sigma\, g^\perp(\cdot, \Phi(\cdot))$ with $\sigma\in\{1, -1\}$. 

Next, using (\ref{Levolve}), we compute the shape operator, for $\id$ denoting the identity on the transverse subspace and $m=\text{dim}_{\mathbb{C}}M-1$, 
\begin{align*}
	L_t &= \frac{H'}{H} \zeta \otimes \eta + \frac{F'}{F}\id.\,\\
{L'}_t &=\Big(\frac{H''}{H}-\big(\frac{H'}{H}\big)^2\Big) \zeta \otimes {\eta}+\Big(\frac{F''}{F}-\big(\frac{F'}{F}\big)^2\Big)\id;\\
\tr{L_t} &=\frac{H'}{H}+(2m)\frac{F'}{F},\\
	\tr{L^2_t} &=\big(\frac{H'}{H}\big)^2+(2m)\frac{(F')^2}{F^2},\\
	\tr{L'_t} &=\frac{H''}{H}+(2m)\frac{F''}{F}-\frac{(H')^2}{H^2}-(2m)\frac{(F')^2}{F^2}.
\end{align*}
By (\ref{detaandhess}) and Theorem \ref{hyperdeformedSasa}, at regular values, $F'\neq 0$ and $(P, \eta_t, \zeta_t, \Phi_t, g_t)$ is a deformed Sasakian structure. Combining (\ref{reducedsolitionsystem}) and Proposition \ref{Rcdeformed}, we have
\begin{align*}
\Rc^\perp -2\frac{H^2}{F^4} g_t &= 	\lambda g_t+\tr(L)g_t\circ L+ g_t\circ L'- f' g_t\circ L_t. 
\end{align*}
Since $g_t= F^2 g^\perp +H^2 \eta\otimes \eta$, one deduces that $\Rc^\perp= k g^\perp$. Consequently, the soliton equation (\ref{grs}) becomes
\begin{align}
	\lambda &= -\frac{H''}{H}-(2m)\frac{F''}{F}+f''\nonumber \\
	&= \frac{H^2}{F^4}(2m)-\frac{H''}{H}-2m\frac{H'F'}{HF}+f'\frac{H'}{H} \nonumber\\
	\label{auxiODE}
	&=\frac{k}{F^2}-\frac{H^2}{F^4}2- \frac{F''}{F}-(2m-1)\big(\frac{F'}{F}\big)^2-\frac{H'F'}{FH}+f'\frac{F'}{F}.
\end{align} 
By Lemma \ref{almostKahler}, the ODE system is augmented with additional constraints {$FF'=\sigma H$ and $H = f'$}. It was investigated and solved explicitly via a transformation to a modified Calabi's ansatz in \cite{DW11coho} (see also \cite{tran23kahler}). 

Explicitly, define $s$ by $2s+A=F^2$ for a constant $A$, and set $\alpha(s):=H^2$. Then $\frac{ds}{dt}=FF'=\sigma H$, so that $dt^2=\frac{ds^2}{\alpha}$ and $\frac{df}{ds}=\sigma$, that is, $f=\sigma s+C$. Hence
\[g=\frac{ds^2}{\alpha(s)}+\alpha(s)\, \eta\otimes\eta+(2s+A)\, g^\perp.\]
Using $F'=\sigma H/F$ and $H'=\frac{\sigma}{2}\frac{d\alpha}{ds}$, a direct computation turns the third equation of (\ref{auxiODE}) into
\[\lambda(2s+A)=k-\frac{d\alpha}{ds}-\frac{2m\alpha}{2s+A}+\sigma\alpha,\]
which is the equation in Theorem \ref{HDcontact} with $B=\sigma$ and $n=m$.

Moreover, $(P, \eta, \zeta, \Phi, \eta\otimes\eta+g^\perp)$ is a deformation of the deformed Sasakian structure on a level set, hence it is itself a deformed Sasakian structure, and by Theorem \ref{chardeformed} it satisfies $(\nabla_X\Phi)Y=b\big(g(X, Y)\zeta-\eta(Y)X\big)$ for a constant $b\neq 0$. By Lemmas \ref{computeA} and \ref{nablaXPhi2}, $d\eta(X, Y)=g(A_X\zeta, Y)=b\, g(X, \Phi(Y))$ for horizontal $X, Y$, so $b=\sigma$. Replacing $(\eta, \zeta)$ by $(\sigma\eta, \sigma\zeta)$, which does not change the metric, we obtain a Sasakian structure (Theorem \ref{classicalSasa}) with $d\eta=g^\perp(\cdot, \Phi(\cdot))$ and $\Rc^\perp=k g^\perp$.

Next, we'll show the global statement. Let $s_0$ be the value of $s$ on $P$. The equation for $\alpha$ is linear, so $\alpha$ extends to a solution on $\{2s+A>0\}$; let $I\ni s_0$ be the maximal interval on which $\alpha>0$ and $2s+A>0$. For $x\in P$ and $s\in I$, let \[\phi(s, x):=\exp_x\big(\tau(s)V(x)\big), \text{ where } \tau(s):=\int_{s_0}^{s}\sigma\, \alpha(r)^{-1/2}\, dr.\] Since the local normal forms around different regular level sets agree on their overlaps, a continuation argument along these geodesics shows that they consist of regular points of $f$ (as $|\nabla f|^2=\alpha>0$) and that $\phi$ is a local isometry from (\ref{calabi}) into $M$. 

It is injective, since $f\circ\phi=Bs+C$ is strictly monotone in $s$ and distinct integral curves of $\nabla f$ are disjoint. Let $Z$ be the critical set of $f$, that is, the zero set of the Killing field $J\nabla f$. By \cite{kobay58}, $Z$ is a union of totally geodesic submanifolds of even codimension; thus, $M\setminus Z$ is connected and dense. The image of $\phi$ is open, and it is closed in $M\setminus Z$: if $\phi(s_j, x_j)\to q\in M\setminus Z$, then $s_j\to (f(q)-C)/B$, which lies in $I$ because 
\[\alpha(s_j)\to |\nabla f|^2(q)>0 \text{ whereas } \alpha\to 0\]
 at the finite end points of $I$ (Lemma \ref{ends}). Following the normal geodesics back to $P$, the points $x_j$ converge to some $x\in P$, and $q=\phi(\lim s_j, x)$. Hence $\phi$ is a diffeomorphism from $I\times P$ onto $M_o:=M\setminus Z$, and each regular level set of $f$ is connected. Moreover, if $q\in Z$, the radial geodesic from $q$ to a nearby regular point meets $Z$ only at $q$, since near $q$ the set $Z$ is a totally geodesic submanifold through $q$. Along this geodesic $f\to f(q)$ while $\tau$ stays bounded, so $f(q)$ corresponds to an end point of $I$ at finite distance. Therefore, the critical level set at such an end point is the closure of the image of the continuous map $x\mapsto \lim\phi(s, x)$, and it is connected. In particular, $f$ is transnormal and each level set is connected.

Finally, we address the ends of the complete manifold. By Lemma \ref{ends} below, $\alpha\to 0$ at the lower end point $s_-$ of $I$, which is a nut, a bolt or a cusp as in (i)--(iii). At an end at finite distance, the $s$-lines converge in $M$ to the critical level set of $f$ there. If $2s+A\to 0$ at that end, all $s$-lines converge to the same point, since \[d\big(\phi(s, x), \phi(s, y)\big)\leq \max\{\alpha, 2s+A\}^{1/2}\, d_{g_P}(x, y)\to 0.\] Hence Lemma \ref{pointsingularorbit} applies in case (i), and Lemma \ref{fibercollapse} applies in case (ii) and at an upper bolt; the result follows.

\end{proof}

For the following results, one assumes the setup as in the proof of Theorem \ref{HDcontact}. In Lemmas \ref{pointsingularorbit} and \ref{fibercollapse}, $t$ is translated, and possibly reflected, so that the end under consideration is at $t=0$ and $t>0$ on $I$. There, $(\eta, \zeta)$ denotes the pair adapted to $t$ as in Lemma \ref{almostKahler}, that is, $J\partial_t=-H^{-1}\zeta$ and $d\eta=\sigma\, g^\perp(\cdot, \Phi(\cdot))$, so that $FF'=\sigma H$ for some $\sigma\in\{1, -1\}$; reflecting $t$ changes the signs of $\eta$, $\zeta$ and $\sigma$, and the Sasakian pair of Theorem \ref{HDcontact} is $(\sigma\eta, \sigma\zeta)$.

\begin{lemma}
	\label{pointsingularorbit}
	Let $I=(0, \epsilon)$ for $\epsilon>0$ and suppose that 
	\[\lim_{t\rightarrow 0^+} H(t)= \lim_{t\rightarrow 0^+} F(t)=0,\]
	and the metric can be extended smoothly to $t=0$. Then, $(P, \zeta, \eta, \Phi, g^\perp+ \eta\otimes \eta)$ is, up to a transverse homothety, the standard Sasakian sphere with Hopf fibration over a complex projective space. Moreover, $\sigma=1$ and $k>0$. 
\end{lemma}

The proof models after one for \cite[Theorem 4.3.3]{pe06book}. The idea is that, as $H, F\rightarrow 0$, the metric must become rounder and rounder. The rigidity of the deformed structure implies that it is actually round.  
\begin{proof}
Let $p$ denotes the point compactification at $t=0$. Since it is locally Euclidean around $p$, each level set of $f$ corresponds to a distance sphere. This also follows from the Morse lemma as \[\He f(p)= \lim_{t\rightarrow 0} f''(t) g.\]

Thus, $P$ is diffeomorphic to a sphere. Moreover, the curves $t\mapsto\phi(t, x)$, $x\in P$, are the unit speed geodesics from $p$. Thus $x\mapsto\frac{d}{dt}\big|_{t=0}\phi(t, x)$ identifies $P$ with the unit sphere $\mathbb{S}^{2n+1}\subset T_pM$, and then $\phi(t, x)=\exp_p(tx)$. Since $g$ is smooth at $p$, in normal coordinates \[g=dt^2+t^2g_{round}+O(t^4), \text{ where } g_{round} \text{ is the round metric of } \mathbb{S}^{2n+1}.\] 
Hence, $t^{-2}g_t\to g_{round}$ as $t\to 0$. As $g_t=H^2\eta\otimes\eta+F^2g^\perp$, evaluating on $\zeta$ and on horizontal vectors shows that $h_0:=\lim_{t\to 0}H/t$ and $\kappa:=\lim_{t\to 0}F/t$ are positive and
\[h_0^2\, \eta\otimes\eta+\kappa^2\, g^\perp=g_{round}.\]
Integrating $(F^2)'=2\sigma H$ gives \[\kappa^2t^2+o(t^2)=\sigma h_0t^2+o(t^2).\] 
Thus, $\sigma=1$ and $h_0=\kappa^2$. By the proof of Theorem \ref{HDcontact}, $(P, \eta, \zeta, \sigma\Phi, g^\perp+\eta\otimes\eta)$ is Sasakian; as $\sigma=1$, so is $(P, \eta, \zeta, \Phi, g^\perp+\eta\otimes\eta)$, and $g_{round}$ is the metric of its $(+, \kappa^2, \kappa)$-deformation. This is a transverse homothety (Remark \ref{scalingdiff}), so by Corollary \ref{varconHF} and Theorem \ref{classicalSasa} the deformed structure is again Sasakian. Its metric $g_{round}$ has all sectional curvatures equal to $1$; in particular, it is a complete simply connected Sasakian manifold of constant $\Phi$-sectional curvature $1$. By \cite{tanno69sasa}, it is the standard Sasakian structure on $\mathbb{S}^{2n+1}$, whose Reeb field generates the Hopf fibration over $\mathbb{CP}^n$. By Proposition \ref{Rcdeformed}, its transverse metric $\kappa^2g^\perp$ has constant holomorphic sectional curvature $4$, so that \[k=2(n+1)\kappa^2>0.\]

\end{proof}

\begin{lemma}
	\label{fibercollapse}
Suppose that 
	\[\lim_{t\rightarrow 0^+} H(t)= 0 \text{ and } \lim_{t\rightarrow 0^+} F(t)\neq 0,\]
	and the metric can be extended smoothly to $t=0$. Then, $(P, \eta, \zeta, \Phi, g^\perp+ \eta\otimes \eta)$ is a Riemannian submersion, with circle fibers the orbits of $\zeta$, over a K\"{a}hler-Einstein manifold $(N, \check{g}_N, J_N)$ with $\Rc_N=k\check{g}_N$. Here $N$ is the critical level set of $f$ at $t=0$, and $F(0)^2\check{g}_N$ is the metric induced on it by $g$. 
\end{lemma}
\begin{proof}
	By the ODE system (\ref{auxiODE}), the set of points corresponding to $t=0$ and $H\rightarrow 0$ corresponds to a level set, called $N$, of a critical value of $f$. $N$ is also the zero set of the Killing vector field $J(\nabla f)$ and, thus, by Kobayashi's \cite{kobay58}, is a totally geodesic submanifold of an even codimension with induced metric $g_N$.
	
	Let's recall the diffeomorphism between nearby regular level sets, induced by the normal exponential map, $\Pi_c^{\epsilon}: M_c \mapsto M_{c+\epsilon}$. It will induce an onto submersion $\Pi^{-\epsilon}_{\epsilon}: M_{\epsilon}\mapsto N=M_0$ for sufficiently small $\epsilon$. Indeed, $M_\epsilon$ is the boundary of the tube of radius $\epsilon$ about $N$: each of its points is joined to $N$ by an $f$-segment of length $\epsilon$, while any curve from $N$ to $M_\epsilon$ crosses all the intermediate level sets and so has length at least $\epsilon$, as $|\nabla t|=1$. Hence, for small $\epsilon$, the $f$-segments starting on $N$ are exactly the normal geodesics of $N$, and for $p\in N$ the pre-image of $p$ under $\Pi^{-\epsilon}_{\epsilon}$ is the image under $\exp_p$ of the sphere of radius $\epsilon$ in the normal space of $N$ at $p$. 
	
	Let $X$ be a vector field on $P$ and its identification on $M_\epsilon$ via $\phi$. Let $\widetilde{X}$ be the push-forward of $X$ via $\Pi^{-\epsilon}_{\epsilon}$. By continuity and Lemma \ref{identification}, 
	\begin{align*} 
		g_N(\widetilde{X}, \widetilde{X}) &= \lim_{t\rightarrow 0}g_t(X, X),\\
		&= F(0)^2 g^\perp (X, X).
	\end{align*}
	Thus, $\widetilde{X}=\vec{0}$ if and only if $X$ is a multiple of $\zeta$. Thus, at each point $q\in M_\epsilon$, the differential of $\Pi^{-\epsilon}_{\epsilon}$ is onto with an one-dimensional kernel. Thus, $N$ is of co-dimension two, the fibers of $\Pi^{-\epsilon}_{\epsilon}$ are circles (the orbits of $\zeta$), and $\Pi^{-\epsilon}_{\epsilon}$ is a homothetic submersion; equivalently, $\pi: (P, g_P)\to (N, \check{g}_N)$, $\check{g}_N:=F(0)^{-2}g_N$, is a Riemannian submersion.

	 It remains to show $(N, g_N)$ is K\"{a}hler-Einstein with a compatible almost complex structure. Indeed, $\Phi$ induces an almost complex structure on $N$ by
	\[J_N \widetilde{X}:=\widetilde{\Phi X}.\]
	\textit{Claim:} $J_N (\widetilde{X})= J(\widetilde{X})$.
	
	\noindent \textit{Proof.} Without loss of generality, one assumes that $X\perp \zeta$ and, by the construction of $\Phi$, $\widetilde{\Phi X} = \widetilde{JX}.$ On the other hand, $\widetilde{X}=\lim_{t\rightarrow 0}(\phi_t)_{\ast} X$ and, by continuity,
	\begin{align*}
		J(\widetilde{X}) &= J (\lim_{t\rightarrow 0}(\phi_t)_{\ast} X)= \lim_{t\rightarrow 0} J ((\phi_t)_{\ast} X) \\
		&=  \lim_{t\rightarrow 0}(\phi_t)_{\ast} (\phi_t^\ast J X)=  \lim_{t\rightarrow 0}(\phi_c)_{\ast} (J X)\\
		&=  \widetilde{JX}.
	\end{align*}
	\textit{Claim:} $(N, g_N, J_N)$ is a K\"{a}hler-Einstein manifold.
	
	\noindent\textit{Proof.} 
	$\Rc^\perp$ is the pull-back of $\Rc_N$ and 
		\[\Rc_N (\widetilde{X}, \widetilde{X}) = \Rc^\perp(X,X)= k g^\perp(X, X)= k \check{g}_N (\widetilde{X}, \widetilde{X})=\tfrac{k}{F(0)^2}\, g_N(\widetilde{X}, \widetilde{X}). \]
	Thus, it is Einstein. Then $J_N$ is parallel since $J_N=J_{\mid TN}$, $J$ is parallel, and $(N, g_N)$ is totally geodesic.

\end{proof}

\begin{lemma}
	\label{ends}
	Write $I=(s_-, s_+)$ and $\alpha'=\frac{d\alpha}{ds}$, and $\Psi(s):=(2s+A)^n e^{-Bs}\alpha(s)$, so that
	\[\Psi'(s)=(2s+A)^n e^{-Bs}\big(k-\lambda(2s+A)\big).\]
	\begin{enumerate}[label=(\alph*)]
		\item At $s_-$, $\alpha\to 0$ and exactly one of the following holds:
		\begin{enumerate}[label=(\roman*)]
			\item (nut) $2s_-+A=0$, $k>0$ and $\alpha=\frac{k}{n+1}(s-s_-)+O\big((s-s_-)^2\big)$; the end is at finite distance;
			\item (bolt) $2s_-+A>0$, $\alpha(s_-)=0$ and $\alpha'(s_-)=k-\lambda(2s_-+A)>0$; the end is at finite distance;
			\item (cusp) $\lambda<0$, $k=\lambda(2s_-+A)\leq 0$ and $\alpha=c\, (s-s_-)^2+O\big((s-s_-)^3\big)$, where $c=|\lambda|$ if $2s_-+A>0$ and $c=\frac{2|\lambda|}{n+2}$ if $2s_-+A=0$; the end is at infinite distance, $s_+=\infty$ and $B<0$.
		\end{enumerate}
		\item Either $s_+=\infty$ and this end is at infinite distance, or $s_+<\infty$, $\alpha(s_+)=0>\alpha'(s_+)$ and this end is a bolt at finite distance. In the latter case the lower end is a nut or a bolt, $\lambda>0$ and $k>0$.
	\end{enumerate}
\end{lemma}
\begin{proof}
	The distance along the $s$-lines is $\int ds/\sqrt{\alpha}$, and $|\nabla f|^2=B^2\alpha$. Since $M$ is complete, at an end of $I$ at finite distance the $s$-lines converge in $M$, so $\alpha$ is bounded near such an end. At a finite end point $s_e$ of $I$ with $2s_e+A>0$ the equation is regular, so $\alpha$ extends smoothly to $s_e$, and $\alpha(s_e)=0$ by the maximality of $I$.

	(a) As $2s+A>0$ on $I$, $s_-\geq -A/2$. Suppose first that $2s_-+A>0$. Then \[\alpha(s_-)=0, \text{ and } \alpha'(s_-)=k-\lambda(2s_-+A)\geq 0.\]
	 If $\alpha'(s_-)>0$, then $\int_{s_-}ds/\sqrt{\alpha}<\infty$, which is (ii). If $\alpha'(s_-)=0$, differentiating the equation gives $\alpha''(s_-)=-2\lambda$. If $\lambda=0$, then $k=0$, the equation is homogeneous and $\alpha\equiv 0$; if $\lambda>0$, then $\alpha<0$ near $s_-$. Hence $\lambda<0$, \[k=\lambda(2s_-+A)<0, ~~\alpha=|\lambda|(s-s_-)^2+O\big((s-s_-)^3\big), \text{ and } \int_{s_-}ds/\sqrt{\alpha}=\infty.\]

	Suppose next that $2s_-+A=0$. As $\Psi'$ is integrable near $s_-$, $\Psi$ extends continuously to $s_-$. If $\Psi(s_-)\neq 0$, then $\alpha=\Psi(2s+A)^{-n}e^{Bs}$ is either negative near $s_-$ or tends to $+\infty$ while $\int_{s_-}ds/\sqrt{\alpha}<\infty$; both are impossible. Hence $\Psi(s_-)=0$ and
	\[\alpha(s)=\frac{e^{Bs}}{(2s+A)^n}\int_{s_-}^s(2r+A)^n e^{-Br}\big(k-\lambda(2r+A)\big)\, dr=\frac{k}{n+1}(s-s_-)+O\big((s-s_-)^2\big).\]
	Moreover, if $k=0$, \[\alpha=-\frac{2\lambda}{n+2}(s-s_-)^2+O\big((s-s_-)^3\big).\] Thus $k\geq 0$. If $k>0$, the end is at finite distance, which is (i). If $k=0$, then, as before, $\lambda<0$ and the end is at infinite distance.

	In both cusp cases \[k-\lambda(2s+A)=2|\lambda|(s-s_-)>0 \text{ for } s>s_-.\] 
	Thus, $\Psi$ increases from $\Psi(s_-)=0$. Hence $\alpha>0$ on $(s_-, \infty)$ and $s_+=\infty$. If $B>0$, then $\Psi$ has a positive limit as $s\to\infty$, so $\alpha$ grows exponentially and $\int^{\infty}ds/\sqrt{\alpha}<\infty$; this contradicts the boundedness of $\alpha$ at an end at finite distance. Thus $B<0$.\\

	(b) If the end $s_+=\infty$ were at finite distance, $\alpha$ would be bounded near it, and then $\int^{\infty}ds/\sqrt{\alpha}=\infty$. If $s_+<\infty$, then \[\alpha(s_+)=0 \text{ and } \alpha'(s_+)=k-\lambda(2s_++A)\leq 0.\] If $\alpha'(s_+)=0$, then, as in (a), $\lambda<0$ and $k=\lambda(2s_++A)$, so that \[\Psi'(s)=-2|\lambda|(2s+A)^ne^{-Bs}(s_+-s)<0 \text{ for } s<s_+.\] Then $\Psi>0$ on $(-A/2, s_+)$, so $s_-=-A/2$ and $\Psi(s_-)>0$, which was excluded in (a). Hence $\alpha'(s_+)<0$, and the end is at finite distance. By (a), the lower end is not a cusp, so $k-\lambda(2s_-+A)>0$, while $k-\lambda(2s_++A)<0$. Therefore,
	\[\lambda>0 \text{ and } k>\lambda(2s_-+A)\geq 0.\]
\end{proof}

\begin{example}
	\label{cuspexample}
	Let $n=1$, $\lambda=-1$, $A=C=0$ and $B=-1$, so that $f=-s$, and let $J\partial_s=-\alpha^{-1}\zeta$ and $J=\Phi$ on $\ker\eta\cap\ker ds$.
	\begin{enumerate}[label=(\arabic*)]
		\item $N=\mathbb{H}^2$ with $g_N=y^{-2}(dx^2+dy^2)$, so \[k=-1;~~ P=\mathbb{H}^2\times\mathbb{R} \text{ with } \eta=d\theta+2y^{-1}dx,\] the standard Sasakian structure on $\widetilde{SL}_2(\mathbb{R})$. $I=(\frac{1}{2}, \infty)$ and $\alpha=s^{-1}p(s)$ with \[p(s)=2s^2-5s+5-3e^{1/2-s}.\]
		\item $N=\mathbb{C}$ with $g_N=dx^2+dy^2$, so \[k=0; ~~ P=\mathbb{C}\times\mathbb{R} \text{ with } \eta=d\theta+x\, dy-y\, dx,\] the Heisenberg group $\mathrm{Nil}^3$. $I=(0, \infty)$ and $\alpha=s^{-1}p(s)$ with \[p(s)=2s^2-4s+4-4e^{-s}.\]
	\end{enumerate}
	In both cases $g^\perp=\pi^\ast g_N$, with $J_N$ chosen so that $d\eta=\pi^\ast\omega_N$. The equation of Theorem \ref{HDcontact} becomes \[p'+p=2s^2-s \text{ and } p'+p=2s^2.\] 
	By direct computation, (\ref{calabi}) is then K\"{a}hler and satisfies (\ref{grs}). Moreover, $p$ and $p'$ vanish at $s_-$ while $p''>0$ on $I$, so $\alpha>0$ on $I$. Also \[\alpha=(s-\frac{1}{2})^2+O\big((s-\frac{1}{2})^3\big), \text{ respectively } \alpha=\frac{2}{3}s^2+O(s^3),\] and $\alpha\sim 2s$ as $s\to\infty$. Hence both ends are at infinite distance and, as the metrics on the slices are complete and uniformly equivalent on compact $s$-intervals, (\ref{calabi}) is a complete expanding K\"{a}hler GRS on $M=I\times P\cong\mathbb{R}^4$. 
	
	Its lower end is a cusp as in Theorem \ref{HDcontact}(iii), with \[2s_-+A>0, \text{ respectively } 2s_-+A=0,\] and $f$ has no critical point. Each level set of $f$ is the $(-, \alpha^{1/2}, (2s)^{1/2})$-deformation of the Sasakian structure $(-\eta, -\zeta, -\Phi, g_P)$. 
	The metric is irreducible 
	and $M$ is a complete, simply connected, irreducible GKRS with $f$ and $\SS$ functionally dependent and without singular orbit. Quotients by discrete groups of automorphisms of $P$ acting freely give further examples, for instance with compact level sets.
\end{example}

\section{Dimension four}
\label{dim4}
We restrict our investigation to real dimension four and great simplification occurs. 
\begin{lemma}
	\label{constantHe} Let $(M, g, J, f)$ be a K\"{a}hler GRS in real dimension four. Suppose that $f$ is rectifiable then along each regular connected component, $\He (f)$ has constant eigenvalues, $\mu_1$ on $\mathrm{span}\{\nabla f, J\nabla f\}$ and $\mu_2$ on its orthogonal complement.
\end{lemma}
\begin{proof}
	Let $M_c$ be a regular connected component of $f$. Since $|\nabla f|$ is constant on $M_c$, for any vector field $X$ tangent to $M_c$, 
	\[\text{Hess (f)}(\nabla f, X)=0.\]
	Following Proposition \ref{rectequivalent}, $b(f)=|\nabla f|^2$ is locally defined for nearby connected components and it is smooth. Thus,  
	\begin{align*}
		\text{Hess (f)} (\nabla f, \nabla f) &=g(\nabla_{\nabla f} (\nabla f), \nabla f)\\
		&= \frac{1}{2}\nabla_{\nabla f}|\nabla f|^2 = \frac{1}{2}b'(c)|\nabla f|^2.		
	\end{align*}
	Therefore, $\nabla f$ is an eigenvector with eigenvalue $\frac{1}{2}b'(c)$. As each tangent space is of dimension $4$ and $\text{Hess f}$ is $J$-invariant, $\text{Hess f}$ has the eigenvalue $\mu_1:=\frac{1}{2}b'(c)$ on $\mathrm{span}\{\nabla f, J\nabla f\}$ and a single eigenvalue $\mu_2$ on its orthogonal complement. By equations (\ref{deltaS}) and (\ref{nablafandS}), 
	\[\Delta f=\tr (\He (f))=n\lambda {+}|\nabla f|^2{-} 2\lambda f- \text{a constant}.\] Thus, it is also fixed along $M_c$ and the result then follows.  
	

\end{proof}

We are now ready to classify all K\"{a}hler GRS with a rectifiable potential function $f$. 

\begin{proof}[Proof of Theorem \ref{main1}]
	Let $M_c$ be a regular connected component of $f$. By (\ref{standardconstruction}),
	
	\noindent $(M_c, \zeta_c, \eta_c, \Phi_c, g_c)$ is an almost contact metric structure. Along a regular connected component, $|\nabla f|$ is a nonzero constant and we have, for $L$ denoting the shape operator,	
	\begin{align*}
		\He (f)(X,Y)&=|\nabla f|g(L X, Y)+\frac{X|\nabla f|}{|\nabla f|}g(\nabla f, Y),\\
		&= |\nabla f|g(L X, Y). 
	\end{align*}
	
	By Lemma \ref{constantHe}, $\He (f)$ has constant eigenvalues $\mu_1, \mu_2$ with $\zeta$ as one of the eigenvector. Therefore,
	\[ L= \frac{1}{|\nabla f|} ((\mu_1-\mu_2)\zeta\otimes \eta+ \mu_2 \id).\] 
	
Consequently, $L$ satisfies the condition of Theorem \ref{hyperdeformedSasa} and $(M_c, \zeta_c, \eta_c, \Phi_c, g_c)$ must be either a deformed Sasakian structure(if $\mu_2\neq 0$) or locally a Riemannian product with $\nabla\Phi_c=0$ (if $\mu_2=0$).

We claim that either $\mu_2$ vanishes identically on the set $M^{reg}$ of regular points of $f$, or it vanishes nowhere on $M^{reg}$. First, $M^{reg}$ is connected and dense. Next, $\{\mu_2\neq 0\}$ is open and also closed in $M^{reg}$. Only the closeness needs justification. Indeed, on an interval of regular level sets on which $\mu_2\neq 0$, the level sets are deformed Sasakian by Theorem \ref{hyperdeformedSasa}, and the computation in the proof of Theorem \ref{HDcontact}, with Lemma \ref{almostKahler}, gives \[g=dt^2+H^2\eta\otimes\eta+F^2g^\perp \text{ with } FF'=\sigma H \text{ and } H=f'=|\nabla f|.\] Hence $\mu_2=f'F'/F=\sigma H^2/F^2$. At a regular level set in the closure of such an interval, $H=|\nabla f|$ and $F$ have positive limits, since $g_t=H^2\eta\otimes\eta+F^2g^\perp$ converges to the metric induced on that level set; so $\mu_2\neq 0$ there as well. Since $M^{reg}$ is connected, the claim follows.

Then, we consider two cases.
	
		\textbf{Case 1:} $\mu_2\neq 0$ and each regular connected component is a deformed structure of a Sasakian model $(P, \eta, \zeta, \Phi, g_P)$. By Theorem \ref{HDcontact}, the transverse K\"{a}hler-Einstein metric $g^\perp$ has constant curvature $k$. By Proposition \ref{Rcdeformed}, $(P, g_P)$ has constant $\Phi$-sectional curvature $k-3$. The regular level sets are closed in the complete manifold $M$, hence complete, and so is $(P, g_P)$. Thus, the universal Sasakian model must be one of the three standard forms in dimension three \cite{tanno69sasa}, modelled on $\mathbb{S}^3$, $\mathrm{Nil}^3$ or $\widetilde{SL}_2(\mathbb{R})$ according as $k>0$, $k=0$ or $k<0$. This argument does not use a fibration of $P$ over a surface, so it also covers case (iii) of Theorem \ref{HDcontact}. The soliton equation (\ref{grs}) then reduces to ODEs which can be solved explicitly \cite{DW11coho, tran23kahler}.

		\textbf{Case 2:} $\mu_2=0$ and each $\zeta_c$ is parallel with respect to the induced metric which is locally a Riemannian product. As in the proof of Theorem \ref{HDcontact}, we write 
		\begin{align*} 
			g &=dt^2+ g_t,\\
			g_t &= H^2(t)\eta\otimes \eta+ g^\perp_t.
		\end{align*} 
		Since $\mu_2=0$, equations (\ref{Levolve}) and (\ref{LandHess}) implies that 
		\[\ddt g^\perp_t =0.\]
		Thus, the metric is a local Riemannian product preserving the almost complex structure
		\[g=(dt^2+ H^2(t)\eta\otimes \eta)+g^\perp.\] By the soliton system (\ref{reducedsolitionsystem}), $g^\perp$ has constant Ricci curvature and, thus, constant curvature (equal to $\lambda$) due to its low dimension. 

Hence, on $M^{reg}$, which is connected by the claim above, the $J$-invariant plane field $\mathcal{D}:=\mathrm{span}\{\nabla f, J\nabla f\}$ and its orthogonal complement are parallel. Let $\tilde M$ be the universal cover, with the lifted structures. Since the critical set has codimension at least two, every piecewise smooth loop in $\tilde M$ based at a regular point can be approximated in $C^1$ by loops through regular points only. As parallel transport depends continuously on the loop, the holonomy group of $\tilde M$ preserves $\mathcal{D}$ at that point. Thus $\mathcal{D}$ extends to a parallel distribution on $\tilde M$, and by the de Rham decomposition theorem \cite[Chapter 10]{besse}, $\tilde M=\tilde\Sigma\times\tilde N$ isometrically, with $T\tilde\Sigma=\mathcal{D}$ at regular points. 

Since $\nabla f\in\mathcal{D}$, the lift of $f$ is a function on $\tilde\Sigma$, and both factors are K\"{a}hler since $J$ preserves $\mathcal{D}$. Hence $\tilde\Sigma$ is a complete $2D$ K\"{a}hler GRS, and $\tilde N$ is a simply connected surface of constant curvature $\lambda$. 
(Complete $2D$ GRS are classified in \cite{BM15}.)


\end{proof}

\begin{proof}[Proof of Corollary \ref{simply}] 

By Theorem \ref{main1}, as $M$ is simply connected and irreducible, the second alternative holds, so Theorem \ref{HDcontact} applies with $n=1$. By Lemma \ref{ends}, the upper end point of $I$ is $+\infty$: otherwise both ends are at finite distance and $k>0$, so by \cite{tanno69sasa} the universal cover of $P$ is $\mathbb{S}^3$, $P$ is compact, and so is $M$. In particular, by Lemma \ref{ends}, the lower end is a nut or a bolt, or else a cusp and $\lambda<0$.

Every isometry $\psi$ of $(M, g)$ preserves $f$. Indeed, $\Rc+\He(f\circ\psi)=\lambda g$, so \[\He(f\circ\psi-f)=0.\] As $M$ is irreducible, $f\circ\psi=f+c$. Applying (\ref{nablafandS}) to $f$ and to $f\circ\psi$ gives $\lambda c=0$. If $\lambda=0$, the lower end is a nut or a bolt, so $f$ has critical points; $\psi$ preserves the critical set of $f$, on which $f$ is constant, and $c=0$. Consequently, isometries preserve the level sets of $f$, and they are determined by their restriction to one regular level set, since they commute with the flow of $\nabla f$.

\textit{Nut or bolt.} Then the critical set of $f$ is a single level set: a point (and then $N=\mathbb{CP}^1$ is the base of the Hopf fibration), or the K\"{a}hler-Einstein curve $N$ of Theorem \ref{HDcontact}(ii). In the latter case, $M$ is diffeomorphic, along the normal geodesics, to the normal bundle of $N$; hence $M$ deformation retracts onto $N$, and $N$ is simply connected.

\textit{Cusp.} Then $\lambda<0$, $f$ has no critical point, and $M\cong\mathbb{R}\times P$, so $P$ is simply connected. As $P$ is complete with transverse curvature $k\leq 0$, by \cite{tanno69sasa} it is, up to a transverse homothety, the standard Sasakian structure on $\mathrm{Nil}^3$ ($k=0$) or on $\widetilde{SL}_2(\mathbb{R})$ ($k<0$); the leaf space $N$ of $\zeta$ is $\mathbb{C}$ or the hyperbolic plane.

In each case $N$ is a simply connected surface of constant curvature, and the lifts to $P$ of its orientation-preserving isometries, together with the flow of $\zeta$, form a four-dimensional group of automorphisms of the Sasakian structure. It acts transitively on $P$, and isometrically on (\ref{calabi}) preserving $f$ and $J$. Hence the orbits of the isometry group are exactly the level sets of $f$: $M$ is of cohomogeneity one, there is exactly one singular orbit in the nut and bolt cases, and all orbits are principal in the cusp case.

{The} isometry group of $N$ is of dimension $2^2-1=3$. That of $(M, g, J, f)$ is then constructed by that of $N$ and the action generated by $J(\nabla f)$; see \cite[Theorem 1.5]{tran23kahler}. 

\end{proof}
\begin{proof}[Proof of Corollary \ref{coho1}]
	It follows from \cite{PW09grsym} that, for $\lambda\neq 0$ and irreducible GRS, the isometry group preserves the potential function $f$. The same conclusion also holds if the group is compact via an average argument as shown in \cite{DW11coho}.
	
	As the structure is of cohomogeneity one, $f$ and $\SS$ have same level sets and are functionally dependent. The result then follows from Theorem \ref{main1}.  
\end{proof}

\subsection{Statements and Declarations:} The work was partially supported by grants from the Simons Foundation [709791], the National Science Foundation [DMS-2104988], and the Vietnam Institute for Advanced Study in Mathematics.


\def\cprime{$'$}
\bibliographystyle{plain}
\bibliography{bioMorse}
\end{document}